\documentclass[10pt. oneside]{amsart}\usepackage[lmargin=1in,rmargin=1in, bmargin=1in, tmargin=1in]{geometry}
\usepackage[dvipsnames]{xcolor}
\usepackage{amsmath, amssymb,tikz}
\usepackage{tikz-cd}
\usepackage{tikz}
\usepackage{todonotes}
\usepackage{mathrsfs}
\usepackage{extarrows}
\usepackage{graphicx}
\usepackage[breaklinks, pagebackref]{hyperref}
\usepackage{IEEEtrantools}
\usepackage{mathrsfs}
\usepackage[utf8]{inputenc}
\usepackage[english]{babel}

\usepackage{comment}
\usepackage{csquotes}

\usepackage[shortlabels]{enumitem}

\tikzset{
labl1/.style={anchor=south, rotate=90, inner sep=1.2mm}
}
\usepackage{nicematrix}  
\usepackage{tikz-cd}
\usepackage{tikz}
\usetikzlibrary{decorations.pathreplacing,matrix,calc,positioning}
\usepackage{nicematrix}
\tikzset{ 
    table/.style={
        matrix of math nodes,
        row sep=-\pgflinewidth,
        column sep=-\pgflinewidth,
        nodes={rectangle,text width=3em,align=center},
        text depth=1.25ex,
        text height=2.5ex,
        nodes in empty cells,
        left delimiter=[,
        right delimiter={]},
        ampersand replacement=\&
    }
}
\usetikzlibrary{decorations.pathreplacing, calligraphy}
\hypersetup{
    colorlinks=true,
    linkcolor=blue,
    filecolor=magenta, 
    citecolor=red,     
    urlcolor=gray,
    pdftitle={explicitdescent},
    pdfpagemode=FullScreen,
    }
\usepackage{tikz}

\makeatletter
\newcommand*{\encircled}[1]{\relax\ifmmode\mathpalette\@encircled@math{#1}\else\@encircled{#1}\fi}
\newcommand*{\@encircled@math}[2]{\@encircled{$\m@th#1#2$}}
\newcommand*{\@encircled}[1]{%
  \tikz[baseline,anchor=base]{\node[draw,circle,outer sep=0pt,inner sep=.2ex] {#1};}}
\makeatother

\DeclareFontEncoding{LS1}{}{}
\DeclareFontSubstitution{LS1}{stix}{m}{n}

\makeatletter
  \newcommand*\textmathversion{\csname textmv@\math@version\endcsname}
  \newcommand*\textmv@normal{m}
  \newcommand*\textmv@bold{b}
\makeatother

\usepackage{dynkin-diagrams}  

\usepackage{mathdots}
\usepackage{mathtools,bm}

\newcommand\givensymbol[1]{
}
\DeclarePairedDelimiterX\Set[1]\{\}{%
  #1%
}
\usepackage{amsmath,amsthm,amssymb}
\usepackage{colonequals}

\newcommand{\QQ}{\mathbb{Q}}

\newcommand{\ZZ}{\mathbb{Z}}

\newcommand{\CC}{\mathbb{C}}

\newcommand{\ch}{\mathrm{ch}}

\newcommand{\Sh}{\mathrm{Sh}}    
\newcommand{\Ab}{\mathbb{A}}   
    
\newcommand{\pr}{\mathrm{pr}}

\DeclareMathOperator{\Gal}{Gal}

\numberwithin{equation}{section}

\newtheorem{theorem}[equation]{Theorem}
\newtheorem{proposition}[equation]{Proposition}
\newtheorem{lemma}[equation]{Lemma}
\newtheorem{corollary}[equation]{Corollary}
\newtheorem*{corollary*}{Corollary}

\newenvironment{theorembis}[1]{%
    \renewcommand{\theequation}{\ref{#1} bis}%
    \begin{theorem}%
}{%
    \end{theorem}%
    \addtocounter{equation}{-1} 
}

\newtheorem{question}[equation]{Question}

\newtheorem*{notation*}{Notation}

\theoremstyle{definition}

\newtheorem{definition}[equation]{Definition}
\newtheorem{problem}[equation]{Question}    

\theoremstyle{remark}
\newtheorem{remark}[equation]{Remark}
\newtheorem{example}{Example}[section]  
    
\newtheorem{notation}[equation]  {Notation}    
\newtheorem{note}[equation]{Note}

\usepackage{tikz-cd}
\usepackage{mathrsfs}

\DeclareFontFamily{U}{wncy}{}
\DeclareFontShape{U}{wncy}{m}{n}{<->wncyr10}{}
\DeclareSymbolFont{mcy}{U}{wncy}{m}{n}
\DeclareMathSymbol{\sha}{\mathord}{mcy}{"58}

\newcommand{\SL}{\mathrm{SL}}

\newcommand{\Gb}{\mathbf{G}}   
\newcommand{\Tb}{\mathbf{T}}

\newcommand{\Hb}{\mathbf{H}}

\newcommand{\Addresses}{{
  \bigskip
  \footnotesize

 \textsc{Department of Mathematics, University of California Santa Barbara, CA 93106-3080}\par\nopagebreak
  \textit{E-mail address}: \texttt{swshah@ucsb.edu}

}}

\usepackage{stmaryrd}

\newcommand{\GG}{\mathbb{G}}

\newcommand{\GL}{\mathrm{GL}}

\newcommand{\et}{\mathrm{\acute{e}t}}

\newcommand{\RR}{\mathbb{R}}

\title{Explicit  Hecke  descent for special cycles}  
\author{Syed Waqar Ali Shah} 
\date{}
\begin{document}

\begin{abstract}  We derive an explicit formula for the action of  a geometric Hecke correspondence on special cycles on a   Shimura variety in terms of    such  cycles  at a fixed neat  level  and  compare it with   another   closely  related  expression sometimes used in literature.  We provide evidence that the two formulas do not agree in general. 
\end{abstract}    

\maketitle

\tableofcontents     

\section{Introduction}  The study of special cycles on Shimura varieties often entails  aspects of  smooth representation theory. In \cite{kudla}, Kudla    considered certain  weighted    linear combinations of  special  cycles on orthogonal Shimura varieties whose equivariance properties are neatly captured by Schwartz spaces admitting a smooth  group action.    Similar constructions have appeared in the works of Cornut \cite{Cornut}, Jetchev \cite{Jetchev}, Li-Liu \cite{LiLiu}, Lai-Skinner  \cite{Skinnerlai} and several other recent works.  
One motivation for studying such equivariant labeling of cycles is to facilitate the study of  Hecke action on them. Although 
conventions differ from one author to another, it seems to us that such descriptions can give rise to two fundamentally different ways of defining  the   said  Hecke action, and only one of these apriori  agrees with the geometric action of Hecke correspondences in all situations. Unfortunately, the `non-geometric' action seems to be occasionally 
 used in literature and assumed to be compatible  with the geometric one. 

The purpose of this note is twofold. First, we derive a general formula for the action of a  Hecke  correspondence on a given irreducible special cycle in terms of such cycles that are all at a fixed (finite) level,   assuming some mild conditions on the level  and  the  Shimura data.  This formula also generalizes to  any  abstract pushforward construction     involving such cycles, e.g., Gysin maps in  \'{e}tale cohomology.    Second, we  address the  aforementioned  compatibility of the two actions on the  group   of equidimensional  algebraic cycles and  provide evidence that they only  agree under very   special  circumstances.  

\subsection{Motivation} We  motivate our question in the  familiar setting of modular curves.   Let $ K $ be a compact open subgroup of $ \GL_{2}(\Ab_{f}) $ that is contained  in a  standard  congruence  subgroup    of level $ N \geq 3 $.    Let  $ Y_{K} $     denote    the      modular curve   of level $ K $ in the sense of \cite{DeligneTS}.         It is  a     smooth quasi-projective  algebraic     curve defined  over   $ \QQ $ with complex points given by    $$ Y_{K}( \CC  )  =  \GL_{2}(\QQ)  \backslash  \big  ( \mathcal{H}^{\pm} \times  \GL_{2}(    \Ab_{f}) / K   \big   )     ,     $$
where $ \mathcal{H}^{\pm}  \colonequals   
 \CC \setminus    \RR  $ and $ \Ab_{f} $ denotes the  ring   of finite rational  adeles.  For $ (x , g)  \in \mathcal{H}^{\pm} \times \GL_{2}(\Ab_{f}) $, we denote by $ [x, g]_{K}  \in Y_{K}(  \CC  )  $ the class of $ (x,g) $.      Let $ \mathcal{C}(K)  \colonequals   \ZZ [ Y_{K}(\CC) ]$ denote   the 
group of  complex   divisors on $ Y_{K}(\CC) $. For
$ \sigma \in \GL_{2}(\Ab_{f}) $, let us denote 
$ K_{\sigma  }  \colonequals K \cap \sigma K  \sigma ^ { - 1}  $  and    $ K^{\sigma}  \colonequals   K \cap \sigma^{-1} K \sigma $ for brevity.    Consider the diagram of $ \QQ$-varieties     
\begin{center}   
\begin{tikzcd}[row sep = tiny]  &           Y_{ K_{\sigma} }   \arrow[dl, "{p}"']       \arrow[r,  "{[\sigma]}", "{\sim}"']     &      Y_{K^{\sigma}    }  \arrow[dr,  "q"]   \\
Y_{K} &   &  &   Y_{K}   
\end{tikzcd}    
\end{center}     
in which $ p, q $ are the natural degeneracy maps induced by the inclusions $ K_{\sigma} \hookrightarrow K $, $ K^{\sigma} \hookrightarrow K $   respectively  and $    [\sigma   ] $ is  the twisting  isomorphism   given on $ \CC $-points via $ [x,g]_{K_{\sigma} }  \mapsto   [x  ,  g   \sigma      ]_{K^{\sigma}} $. The maps $ p $ and $ q $ are finite \'{e}tale and 
induce pushforward and pullbacks on divisors with expected  properties.      The $ \ZZ $-linear map $   [K \sigma K ]_{*}  =   q_{*}  \circ  [\sigma]_{*} \circ p ^{*}  \colon  \mathcal{C}(K) \to \mathcal{C}(K) $  is called 
 the  \emph{covariant Hecke correspondence}\footnote{This is sometimes referred to as  ``Albanese"    convention for Hecke corrrespondences  \cite[p.\ 443] {Ribet}. See also  \cite[Remark 4.1.1, 4.3.2]{LLZmodular} and \cite[\S 2.6]{LLZHilb} for a  discussion.}  induced by $ \sigma $. It is  easily  seen  that 
\begin{equation}   \label{CM1}       [ K \sigma K ]_{*}  \cdot  [x,g]_{K}  =  \sum 
_{ \gamma \in K \sigma K / K }  [x, g  \gamma ]_{K}   . 
\end{equation}    On  the  other hand, one can  construct  the  inductive  limit $$  \widehat{\mathcal{C}}  \colonequals 
 \varinjlim  \nolimits _{  L }   
\mathcal{C} (L)   $$ over all (sufficiently small)  levels   $ L  $  of   $  \GL_{2}(\Ab_{f})  $, where the transition maps for the limit are given by   pullbacks   along the  degeneracy maps $ Y_{L'} \to Y_{L} $ induced by  the   inclusions $ L' \hookrightarrow L $.        
It comes   equipped  with a smooth  left  action   
$$    \GL_{2}(\Ab_{f})  \times   \widehat{  \mathcal{C}}   \to   \widehat{\mathcal{C}}   $$        
given by pullbacks along   the     twisting isomorphisms.   Explicitly, if $ z \in  \widehat{\mathcal{C}} $  can be represented by $ [x, g]_{L} \in \mathcal{C}(L)$, then  $ h \cdot z $ for $ h \in  \GL_{2}(\Ab_{f})   $ is the class of $ [x ,g  h^{-1}]_{ h   L  h   ^ { - 1 } } $ in $ \widehat{\mathcal{C}}   $.  Since  divisors satisfy  \'{e}tale descent,  one  has $   \mathcal{C}(K) = \widehat{\mathcal{C}}^{K} $, the equality obtained via the  inclusion $  \mathcal{C}(K) \hookrightarrow  \widehat{\mathcal{C}}^{K}   $. Under this identification,  the pushforward $ \mathcal{C}(L) \to \mathcal{C}(K) $ along the    degeneracy map induced by the inclusion of any (and not necessarily normal) subgroup $ L$ of $ K $ corresponds    to  the norm map $ \sum \nolimits   _    { \gamma \in K / L }  \gamma $.     
In particular, $ q_{*}  $ is identified with $ \sum _{  K / K^{\sigma} }  \gamma $ and $ [K \sigma K]_{*} $  with    $ \sum_ { \gamma \in K/ K^{\sigma} } \gamma \sigma^{-1} $.  Thus  
\begin{equation}
\label{CM2}     [K \sigma K]_{*}  \cdot [x,g]_{K}  =   \sum _{  \delta  \in K \backslash K \sigma K }  [x, g \delta ]_{ \delta ^{-1} K \delta } .    \end{equation}
where the right hand side is interpreted as an element of $  \widehat{ \mathcal{C}} ^{K}   $.  Note  that  individual  points in this    expression     live on  modular  curves of \emph{different}    levels. 

The RHS of (\ref{CM2}) computing the action of $[K \sigma K]_{*}$ is formal in nature and remains valid, for instance,  in the cohomology of any Shimura variety and, more generally, any locally symmetric space.   The RHS of (\ref{CM1})  is however more useful, since it explicitly gives a divisor in terms of points that all share the  original   level  $ K $. It is also what one usually finds in   the   literature on Shimura curves; see, for example, \cite[\S 3.4]{CorVastal} and \cite[\S 1.4]{Shouwu}. We say that (\ref{CM1}) provides an \emph{explicit descent formula} for the formal $K$-invariant expression given by (\ref{CM2}).

\subsection{Main problem}  We may make similar considerations for special cycles on Shimura varieties.   In the preceding discussion, say   $ E $ is an imaginary quadratic    field and $ x=   x_{\iota} $ is defined  as the image of `$ \mathrm{pt} $'     under the  morphism of Shimura  data    
\begin{equation}   \label{CM4}         \iota \colon   (  \mathrm{Res}_{E/\QQ} \GG_{m} , \left \{ \mathrm{pt} \right \}   ) \hookrightarrow  (\GL_{2}, \mathcal{H}^{\pm} )   
\end{equation}  
obtained by considering $ E $ as a $ \QQ $-vector space.  Then the points $ z_{L}(g)   \colonequals     [x_{\iota}, g]_{L}    \in  Y_{L}(\CC)  $  for $ g \in   \GL_{2}(\Ab_{f}) $ are algebraic and referred to as special points (or CM points). One may equivalently describe $ z_{L}(g) $ as the point obtained by taking the distinguished geometric connected component `$ [1]  $'    of the zero-dimensional  Shimura variety for  $ \mathrm{Res}_{E/ \QQ} \GG_{m} $ of    level  $ (\mathrm{Res}_{E/ \QQ} \GG_{m})(\Ab_{f}) \cap gLg^{-1}$, embedding it into $ Y_{g L g^{-1}} $ and taking its image under  the     twisting   isomorphism        $ [g]  \colon  Y_{g L g^{-1}} \xrightarrow{    } Y_{L} $. See  \cite[\S 2]{Norm} for details on this setup.         In  this      paradigm, we may replace   (\ref{CM4})  by     an arbitrary morphism  of  Shimura  data  
$$ \iota  \colon    (\mathbf{H}, X_{\Hb}) \hookrightarrow (\Gb, X_{\Gb})  $$ 
and construct, for any  compact open subgroup $ L $ of  $    \Gb ( \Ab_{f} ) $, \emph{irreducible    special      cycles}  $ z_{L}(g)  $ (defined over the algebraic  closure $ \overline{\QQ} \subset \CC   $ of $ \QQ $)  of a fixed codimension $n $  on $ \mathrm{Sh}_{\Gb}(L) $   for any $ g \in  \Gb(\Ab_{f}) $   in an  analogous  fashion. Let $ \mathcal{C}^{n}(L) $  denote  the  free $\ZZ$-module on codimension $n$ irreducible $ \bar{\QQ} $-subvarieties  on $ \Sh_{\Gb}(L) $.    
One may  then  ask  for  the analog of (\ref{CM1}) for  the  Hecke action on the special cycle $ z_{L}(g) $ in the space $ \mathcal{C}^{n}(L) $. 

More precisely, let $ K $ denote  a fixed compact open subgroup of $ \Gb(\Ab_{f}) $. To avoid pathologies arising from degrees of degeneracy maps,    we assume that $ K $ is \emph{neat} \cite[\S 0.1]{PinkThesis}.   
Fix   a $ \sigma \in \Gb(\Ab_{f} )$ and  denote $ K_{\sigma}   \colonequals    K \cap \sigma K \sigma^{-1} $ and  $ K^{\sigma} \colonequals  K \cap  \sigma^{-1} K  \sigma $ as before.  We  define the \emph{covariant Hecke correspondence} 
\begin{equation*}  [  K  \sigma   K  ] _{*}  \colon     \mathcal{C}^{n}(K) \xrightarrow{ \pr^{*} }  \mathcal{C}^{n}( K_{\sigma} 
 )  \xrightarrow{ [\sigma]_{*} }  \mathcal{C}^{n}(K^{\sigma} )    \xrightarrow{\pr_{*} }   \mathcal{C}^{n}(K),   
 \end{equation*}    
 where $ \pr^{*} $ (resp.,  $ \pr_{*}$) denotes the flat pullback (resp.,    proper pushforward) of cycles along the degeneracy map induced by the inclusion $ K_{\sigma}  \hookrightarrow  K $ (resp.,  $K^{\sigma}  \hookrightarrow  K $)  and   $ [\sigma]_{*} $ is the isomorphism induced by pushforward along  the twisting map\footnote{Throughout, the action of $  \sigma \in \Gb(\Ab_{f}) $ on the tower of Shimura varieties is denoted by $ [\sigma] $ and gives a  \emph{right action} of $ \Gb(\Ab_{f}) $, as in  \cite[Definition 5.14]{Milne}.} $ [\sigma]  \colon  \Sh_{\Gb}(K_{\sigma}   ) \to  \Sh_{\Gb}( K^{\sigma}  ) $.    Cf.\   \cite[\S 6.2]{Noot} and  \cite[\S 1.2]{Esnault}.      In analogy with (\ref{CM2}), it  is  not hard to  show that for any $ g \in \Gb(\Ab_{f})$, we have           
 \begin{equation*}    
 [K \sigma K]_{*} \cdot  z_{K}(g)   =    \sum_{\delta \in K \backslash K \sigma K } z_{\delta^{-1} K \delta }( g \delta ),
 \end{equation*}  where the RHS is   again     viewed as an element in the inductive limit of $ \mathcal{C}^{n}(L) $ over all (sufficiently small) levels $ L $.           A natural question that arises at this point is the following.
\begin{problem}   \label{mainquestion} Is   there  a $ \ZZ $-linear combination of  irreducible   special  cycles  $ z_{K}(-)  \in    \mathcal{C}^{n}(K) $     that equals  $  [K \sigma K]_{*}  \cdot  z_{K}(g) $?  Equivalently, is it possible to write  the   $ K  $-invariant   
   limit expression 
$\sum   \nolimits
_ { \delta \in  K \backslash  K \sigma K }  z_{ \delta ^{-1}  K  \delta } ( g \delta ) $ in terms of such cycles? If so, can one give a formula for it?       
\end{problem} 
A first   guess  suggested by  (\ref{CM1})  would  be that  
\begin{equation}   \label{falsezkg}   [K \sigma K]_{*}  \cdot  z_{K}(g)     \stackrel{?}{=}  \sum_{ \gamma \in K \sigma K  /  K    } z_{K}(g \gamma ) . 
\end{equation}   
Unfortunately, this is not  necessarily  always the case. While the answer to Question \ref{mainquestion} is affirmative, the correct expression is given by Corollary \ref{zkgcoro} below and  involves various non-trivial coefficients. In fact, the number of irreducible special cycles   
in our expression seems    to be far less in general than what is prescribed by (\ref{falsezkg}). See \S \ref{mainfalseexamplessec} for computations.    This   discrepancy      
   continues to exist   if one replaces $ z_{K}(g) $ with the full fundamental cycle of the source Shimura variety.     Analogs of   (\ref{falsezkg})  seem to be used in  \cite[\S 3.1]{JetchevBoum}, \cite[p.\ 846]{LiLiu} and   various       other writings.     

On the    other  hand,     we show that for the cases of zero cycles and codimension zero cycles, our expression does agree with (\ref{falsezkg}). This is of course in agreement with (\ref{CM1}). For a simple group theoretic explanation of this   phenomenon, we refer the reader to \S \ref{Schwartzsection}.

Our  approach to   Question  \ref{mainquestion}       is based on an elementary  but   useful  idea  that explicit distribution relations among objects associated with $ \mathrm{Sh}_{\Hb} $ (such as fundamental cycles of $ \mathrm{Sh}_{\Hb} $)    can be   parlayed  for corresponding relations on $ \mathrm{Sh}_{\Gb} $ via  a  gadget  we refer to as    \emph{mixed Hecke correspondence} \cite[\S 2]{AES}. This approach  utilizes the built-in functoriality properties of pushforward and pullbacks of cycles  on   schemes. Consequently,  our methods readily  apply in  the    more  general  settings of   pushforwards  of  cycle  classes 
into  cohomology  with  coefficients in  local systems on $ \mathrm{Sh}_{\Gb} $, such  as  the  one    studied in  \cite{Anticyclo}.   This is a  non-trivial extension, since cohomology with integral coefficients does not satisfy Galois descent in general. 
We establish this formula   in     Theorem \ref{descent} for abstract pushforwards using the language of RIC functors developed in \cite[\S 2]{AES}, which we briefly recall in \S \ref{RICfuncsec}.  We then derive the expression asked for in Question \ref{mainquestion} in  Corollary  \ref{zkgcoro}.   
We anticipate that our formula will find applications in establishing Euler system   norm  relations when one works with more general coefficients   systems  as,  for instance,       in the approach   originally    envisioned by Cornut \cite{Cornut}.

\subsection{Outline} This note is structured as follows. In \S  \ref{recollections}, we  collect relevant facts  about Shimura varieties needed  in the derivation of our formula.  In \S \ref{pushforwardsofcycles}, we study the special cycles $ z_{K}(-) $ from the introduction and their analogs  for  abstract pushforwards.   In \S \ref{theformula}, we derive our main formula  and   provide some simple  examples.    In \S   \ref{mainfalseexamplessec}, we furnish two sets of counterexamples to  (\ref{falsezkg}).  In \S \ref{Schwartzsection}, we investigate an analogous question in the setting of function spaces that  provides another perspective on the issue.  This  section   may  be  read  independently of the rest of this  note.

\subsection{Acknowledgments}The author is grateful to Christophe Cornut for several discussions and to Barry Mazur for valuable feedback on an earlier draft. The author also thanks the anonymous referees for their diligent reading of the manuscript and for numerous helpful suggestions that significantly improved the exposition of this note.
\section{Recollections}

\label{recollections}

In this section, we recall some basic facts  about Shimura varieties that we will need later on.   We also  generalize a result from \cite{LSZ}, which  provides  a simple criteria  for checking when maps between Shimura  varieties are closed immersions. It is however only  needed  in   \S   \ref{mainfalseexamplessec}. 
\subsection{Shimura varieties} \label{shimurasetup}    Suppose  that    
\begin{equation} \iota \colon  (\mathbf{H}, X_{\mathbf{H} }  )      \hookrightarrow  ( \mathbf{G} ,  X_{\mathbf{G}   }    ) 
\end{equation}    is   an      embedding  of Shimura   data, i.e.,  $ \mathbf{H} $, $ \mathbf{G} $ are   connected  reductive   algebraic  groups over $ \QQ $, the pairs $ (\mathbf{H},X_{\mathbf{H}}) $, $ (\mathbf{G},   X_{\mathbf{G}}) $  satisfy Milne-Deligne axioms (SV1)-(SV3)   \cite[Definition 5.5]{Milne} and $ \iota  \colon  \Hb \to \Gb $ is an embedding preserving   the       said datum.    Denote by $ \Ab_{f} $ the  ring    of finite   rational  adeles and  set     
$$ H \colonequals \Hb(\Ab_{f}),   \quad   \quad              G  \colonequals \Gb(\Ab_{f} )  .   $$      For each neat (\cite[\S 0.1]{PinkThesis}) compact open subgroup $ K $ of $  G    $, there is a smooth   equidimensional     quasi-projective Shimura variety $ \Sh_{\Gb}(K) $ defined    over a canonically defined subfield of $ \CC   $    called the reflex field of $ \Gb $.  Similarly for $ \Hb $.   We will however exclusively consider  all     Shimura varieties as objects in $\mathbf{Sch}_{\overline{\QQ} }$, the category of schemes over the algebraic closure of $ \QQ $ in $ \mathbb{C} $.         

The $ \CC $-points of $ \mathrm{Sh}_{\Gb}(K)   $    are given by $ \Gb(\QQ) \backslash (X _{\Gb} \times G/K) $ and a general point is denoted as $ [x, g]_{K} $ where $ x \in X_{\mathbf{G}} $  and    $ g \in G $. Similar notations and conventions  will be used for corresponding objects  associated  with  $ \Hb $.     We set  
\begin{equation}   \label{codim}     n  \colonequals \dim \Sh_{\Gb}(K) - \dim \Sh_{\Hb}(U)   
\end{equation}  for  some  choice  of  compact open subgroups $  K \subset G $  and $ U  \subset H $. Then $ n $ is independent of $ K$  and $ U $. In what follows, we    think of $ H $ as a closed  subgroup of $ G  $. Then  for any compact open subgroup $ K $ of $ G $, the intersection     $ H \cap K    \colonequals H \cap \iota^{-1}(K)  $ is a compact open subgroup of $ H $,  which is   neat if $  K $ is.         

\subsection{Maps} 
For any $ g \in G $ and $  K $ a neat  compact open subgroup of $ G $,  there is  a   \emph{twisting isomorphism}       
\begin{equation}   \label{twistingiso}       [g]  =  [g]_{K, g^{-1}Kg }    \colon  \mathrm{Sh}_{\Gb}(K)   \xrightarrow{\sim}      \mathrm{Sh}_{\Gb}(g^{-1} K g )   
\end{equation}    
given on $ \CC $-points by $ [x, \gamma]_{K}  \mapsto  [x,  \gamma g ] _ { g ^{-1} K g }  $ for all  $ x \in X_{\Gb} $ and $   g     \in  G $. 
Let $ \mathbf{Z}_{\Gb} $ (resp.,  $  \mathbf{Z}_{\Hb}  $) denote the center of $ \Gb $ (resp.,  $ \Hb $). For any compact open subgroup $ L $ contained in $  K $  that satisfies 
\begin{equation}     \label{dagger} K \cap  \mathbf{Z}_{\Gb}(\QQ) = L \cap \mathbf{Z}_{\Gb}( \QQ  ),    \end{equation} 
the natural  surjection   $ \pr_{L,K} \colon \mathrm{Sh}_{\Gb}(L)  \to \mathrm{Sh}_{\Gb}(K) $ given by $ [x, g]_{L} \mapsto [x,g]_{K} $ is finite \'{e}tale  of degree $ [K:L] $  by \cite[Lemma 2.7.1]{AES}.  When (SV5) of \cite{Milne} holds for $ (\Gb, X_{\Gb}) $, the condition (\ref{dagger}) is automatic by neatness of $ K $.  
One may also consider the limit $$ \Sh_{\Gb}(\CC) \colonequals  \varprojlim   \nolimits  _ {   L } \Sh_{\Gb}(L)(\CC) $$ taken over all    neat      compact open subgroups $ L$ of  $  \Gb(\Ab_{f})  $ along the  degeneration maps induced by the  inclusions     $  L '  \hookrightarrow   L  $.        Then $ \Sh_{\Gb}(\CC) $ inherits a  continuous right action of $  G   =   \Gb(\Ab_{f}) $ and  $  \Sh_{\Gb}(K)(\CC) = \Sh_{\Gb}(\CC) / K $ for any $ K $. If $ (\Gb, X_{\Gb}) $ satisfies (SV5),    any neat compact open  subgroup  acts  freely on $ \Sh_{\Gb}(\CC)  $.

We also have maps between  Shimura  varieties  of  $ \Hb $ and $ \Gb  $.  If $ K $ is a neat compact open subgroup of $ G $ and   $ U  $ is a compact open subgroup of $ H \cap K $, 
 there is a finite unramified   morphism   
 \begin{equation} \label{iotaUKmap}  \iota_{U,K}  \colon  \mathrm{Sh} _{\Hb} ( U )  \to   \Sh    _{\Gb}(K)   \end{equation}    
 of smooth schemes  in    $ \mathbf{Sch}_{\overline{\QQ} } $,  which is  given on $ \CC $-points   by     $  [y, h]_{U}  \mapsto [y, h]_{K} $ for all $ y \in X_{\Hb} $ and    $  h \in H  $.    We can   always   choose a  compact open subgroup $ L =  L_{U} $ of $ G $  which contains  $ U $ such that $  \mathrm{Sh}_{\Hb}(U) \to \mathrm{Sh}_{\Gb}(L ) $ is a closed immersion  \cite[Proposition 1.15]{DeligneTS}. We may assume   that $  L \subset K $ by replacing $ L $ with  its intersection with  $ K $ if necessary, since $ \mathrm{Sh}_{\Hb}(U) \to \mathrm{Sh}_{\Gb}(L) $ factors as $$ \mathrm{Sh}_{\Hb}(U)  \to     \mathrm{Sh}_{\Gb}(L \cap K ) \to \mathrm{Sh}_{\Gb}(L)   $$   
and \cite[\href{https://stacks.math.columbia.edu/tag/07RK}{Lemma 07RK}]{stacks-project} implies that $ \mathrm{Sh}_{\Hb}(U) \to \mathrm{Sh}_{\Gb}(L \cap K )$ is  also  a     closed immersion. We may moreover assume    that  (\ref{dagger}) holds for $ L $. Indeed  if we  let   $ L'   \colonequals  L \Gamma $ where $ \Gamma \colonequals \mathbf{Z}_{\Gb}(\QQ) \cap K $,   we  have      $ L \subset L'  \subset K $ 
and it is  elementary to show that $ \mathrm{Sh}_{\Gb}(L) = \mathrm{Sh}_{\Gb}(L') $.     In any case,  $ \iota_{U, K } $ is a  composition of a closed  immersion  followed  by  a  finite  \'{e}tale  surjection and one may   therefore 
apply  various   natural     pushforward  constructions in this setting (see e.g.,  \cite[Proposition A.5]{Anticyclo}). In all cases, the natural map \begin{equation}
\label{infty}    \iota_{\infty}  \colon \Sh_{\Hb}(\CC)  \to  \Sh_{\Gb}(\CC) \end{equation}    at   the  infinite    level     is  injective, again  by \cite[Proposition 1.15]{DeligneTS}.    The  following    generalization  of  \cite[Proposition  5.3.1]{LSZ} gives a  criteria for immersion at finite  levels.   

\begin{lemma}         Assume  that $  ( \Gb, X_{\Gb} ) $  satisfies (SV5) and that there exists a $ w \in \mathbf{Z}    _{\Hb}(\QQ) $ such that $ \mathbf{H} $ is the centralizer of $ w $ in $ \Gb $. 
  Let $ K  $  be a   compact open subgroup of $ G $ such that there exists a   neat    compact open subgroup of $ G $ that contains both $ K $ and $ w K w^{-1} $.    Then $ \iota_{  H \cap K , K  } $ is a closed  immersion.   \label{closedimmersion}    
\end{lemma}  

\begin{proof}   Since  $  \iota _ { H \cap K , K } $ is a finite unramified morphism, it  suffices by \cite[\href{https://stacks.math.columbia.edu/tag/04XV}{Lemma 04XV}]{stacks-project} to check that $ \iota_{H \cap K , K} $ is universally injective. By \cite[\href{https://stacks.math.columbia.edu/tag/01S4}{Lemma 01S4}]{stacks-project}, this amounts to  the surjectivity of the diagonal map $$ \Delta  :   \Sh_{\Hb}(H \cap K ) \to   \Sh_{\Hb}(H \cap K ) \times_{\Sh_{\Gb}(K) } \Sh_{\Gb}(H \cap   K )  .  $$ 
By  \cite[I, \S3, 6.11]{DG}\footnote{See  also the notion of \emph{ultraschemes} in  \cite[Appendice]{EGAI}.},             
one can check surjectivity on closed points and therefore,  also    on      $ \CC $-points. But it is easily seen that   $ \Delta(\CC) $ is surjective if and only if $ \iota_{ H \cap K , K } (\CC) $ is injective. So it suffices to check the injectivity of $ \iota_{H \cap K , K  } $ on $  \CC $-points.

To this   end, say there exists   a     $ P \in \Sh_{\Hb}(\CC) $ and $  \kappa \in K $   such that $  P \kappa $ lies in $  \Sh_{\Hb}(\CC) $,  where  we are viewing $ \mathrm{Sh}_{\Hb}(\CC) $ as a subset of $ \Sh_{\Gb}(\CC) $ via (\ref{infty}).  Since $ w   \in \mathbf{Z}_{\Hb}(\QQ)   $     fixes $ X_{\Hb} $ pointwise and commutes with  all  elements of $ \Gb(\Ab_{f})$, the (right) action of $ w $ on $ \Sh_{\Gb}(\CC) $ fixes $ \mathrm{Sh}_{\Hb}(\CC) $ pointwise.  Therefore, 
$$  P w  \kappa =  P  \kappa    = P  \kappa w $$     This  implies  that   $  \lambda  \colonequals  \kappa w \kappa^{-1} w ^{-1}  $  stabilizes $ P $.   If $ L $ is a compact open subgroup of $ G $ that contains both $ K $ and $ w  K w ^{-1} $, then $ \lambda =  \kappa  (w  \kappa ^{-1} w ^{-1} ) \in L $.  Recall that  if $ (\Gb, X_{\Gb}) $ satisfies (SV5), then the right action of any neat compact open subgroup of $ G $ on $ \Sh_{\Gb}(\CC) $ is   free.     Thus   
if $ L  $ is neat, the only element of $ L $ that can stabilize  $ P $ is  the   identity. So our assumptions imply that  $  \lambda  $ must be identity,  i.e., $ w =  \kappa w \kappa ^{-1} $. Since the centralizer of $ w $ in $ \Gb $ is $ \Hb $, we have     $  \kappa   \in H $.  Therefore $   \kappa    \in H \cap K $ and  the  classes of $ P $ and $ P  \kappa   $ in  $  \Sh_{\Hb}(  H  \cap   K    )(\CC) $ are  forced to be  equal.  This  implies   that     $ \iota_{H \cap   K    ,     K} $  is     injective   on   $ \CC  $-points  and is therefore  a   closed  immersion.  
\end{proof}    

\begin{remark}  While we exclusively work in $ \mathbf{Sch}_{\overline{\QQ} } $,  many  of our results   hold true over canonical models over the reflex field of $ \Hb $. For instance,  (\ref{iotaUKmap}) is defined over this field and the criteria of Lemma \ref{closedimmersion} applies if $ \iota_{H \cap K, K} $ is viewed over this number  field.    
\end{remark} 

\subsection{Connected components}   \label{conn}  Recall that we are viewing all Shimura varieties  in   $ \mathbf{Sch}_{\overline{\QQ}}  $.        To  simplify  the  description  of  the  connected components of the Shimura varieties   of $ \Hb $, we  assume for the rest of this note   that    the derived group $ \mathbf{H}^{\mathrm{der}} $ is simply  connected.   Let $ \mathbf{T} = \mathbf{H} / \mathbf{H}^{\mathrm{der}} $ and  $ \nu \colon \mathbf{H} \to \mathbf{T} $  denote    the natural map. Let $ \Tb(\RR)^{\dagger}  \subset \mathbf{T}   (   \RR    )      $ denote the image of the real points of center of $ \mathbf{H} $ in the real points of $ \mathbf{T} $  and  let   $  \Tb (  \QQ  )   ^  {  \dagger}  \colonequals      \Tb ( \QQ)  \cap  \Tb ( \RR ) ^ { \dagger} $.    Let $ U  $   be a   neat    compact open subgroup of  $  H   $.   Then $ \nu(U) $ is a compact open subgroup of $ \Tb(\Ab_{f})   $   by  \cite[Lemma 5.21]{Milne} and  neat by \cite[Corollary 17.3]{Borel}.     The set of    connected components of $ \mathrm{Sh}_{\Hb}(U) $ can be described  via  an   explicit    bijection (of sets)   
\begin{equation}   \label{pi0}    \pi_{0} (\Sh_{\Hb}(U))(\CC)      \xrightarrow{\sim}   \Tb(\QQ)^{\dagger}  \backslash \Tb(\Ab_{f} )  /  \nu(U )  
\end{equation} 
once a connected component $ X_{\Hb}^{+} \subset X_{\Hb} $ in the analytic topology of $ X_{\Hb} $ is fixed (\cite[Theorem 5.17]{Milne}),   which we  do  in  what   follows. Since $ \pi_{0}( \Sh_{\Hb}(U)) $ is a finite \'{e}tale scheme and we are   working     over $ \overline{\QQ}$,  there is no harm in identifying  it  with   $$ \pi_{0}(\Sh_{\Hb}(U))(\overline{\QQ}) \xrightarrow{\sim} \pi_{0}(\Sh_{\Hb}(U))(\CC)  ,   $$  and we will merely view these as sets without a  scheme  structure.   We   however   do    view  $ \pi_{0}( \Sh_{\Hb}(U)) $ as an abelian group  by transport of   structure via  (\ref{pi0}), and will denote this abelian  group by $ \pi_{0, U} $ for the rest of this note for  simplicity.      
\begin{definition}   \label{Ycirc}  By $ Y_{U}^{\circ} $, we denote  the   connected component of $ \mathrm{Sh}_{\Hb}(U)$ over the  identity element  $ [  1   ] 
\in \pi_{0, U}    $ which we consider  as an object in $ \mathbf{Sch}_{\overline{\QQ}} $.     If $ V $ is a compact open subgroup of $ U $, we denote by $ \varphi_{V,U} \colon \pi_{0, V} \to  \pi_{0 ,  U  } $ the  quotient  homomorphism.  
\end{definition}  

For convenience, we will write $ Y^{\circ} $ for $ Y_{U}^{\circ}    $  when  the  level is  apparent  from  context.  
We observe that the connected components of $ \Sh_{\Hb}(U) $ are equidimensional of  dimension $ \dim \,  \Sh_{\Hb}(U) $. It is easy to see that the twisting  isomorphism $ [h] \colon \Sh_{\Hb}(U)  \to  \Sh_{\Hb}(h^{-1} U h  ) $ sends $ Y_{U}^{\circ} $ to the 
connected component of  $ \Sh_{\Hb}(h^{-1} U  h ) $    indexed by the class of $ \nu(h) $ in $   \pi_{0, h^{-1} U h } =   \pi_{0,U}      $. 
\begin{remark} We note in passing that each component in $ \pi_{0, U}$  is    defined over a finite abelian extension of the reflex field, which  is determined by   an  explicit       Shimura-Deligne reciprocity law, and the eventual     expression we derive for the Hecke action in Question  \ref{mainquestion} is defined over the compositum of the fields of definition of cycles involved.   
This is  however not relevant to the issues  considered in this  note.       
\end{remark}  

\section{Abstract pushforwards}

\label{pushforwardsofcycles}    
In this section, we  carefully define the cycles $ z_{K}(-) $ from the introduction and consider their analogs    in the more general setting of abstract pushforwards, which model the behaviour of any suitable cohomology theory for Shimura varieties.     

\subsection{Special cycles}   \label{specialcycles}    
We maintain  the  notations and assumptions  of  \S  \ref{recollections}.  For each integer $ m \geq 0 $ and neat compact open subgroup  $ K $ of $ G $, we denote by $ \mathcal{C}^{m}(K) $ the free abelian group on the set of  codimension $ m $ closed integral $\overline{\QQ}$-subschemes of  $\mathrm{Sh}_{\Gb}(K) $.     
We refer to elements of $ \mathcal{C}^{m}(K)$ as \emph{algebraic cycles of codimension $ m $ on $ \mathrm{Sh}_{\Gb}(K)$}.  Recall  (\ref{codim}) that $ n $ denotes the codimension of Shimura varieties  of $ \Hb $ in those of $ \Gb   $.  Recall  also    that $ \iota_{H \cap K , K } $ (\ref{iotaUKmap}) is finite and therefore proper. In particular, $ \iota_{H \cap K , K } $ is a  closed map. 
\begin{definition} \label{zkg} For  $ g \in G $,  we denote by  $ z _{K}(g) \in  \mathcal{C}^{n}(K) $ the algebraic cycle  on 
$   \mathrm{Sh}_{\Gb}(K)  $    given by the fundamental cycle  of the reduced closed  subscheme  of  $ \Sh_{\Gb}(K)  $    whose underlying topological space is given by the (necessarily    irreducible and closed) image of   $ Y_{H \cap  g K g ^{-1} }^{\circ} $ under the  morphism     
\begin{equation}   \label{Cornuttwist}
\Sh _{\Hb}( H \cap  g K  g ^{-1} )  \xrightarrow { \iota    }        \Sh_{\Gb}( g K  g    ^{-1})  \xrightarrow{ [g] }   \Sh_{\Gb}(K) .  
\end{equation}  We   refer to these as  \emph{irreducible special cycles}. 
We denote by $ \mathcal{Z}_{K} $ the $ \ZZ $-submodule of $ \mathcal{C}^{n}    (  K   )      $ spanned by 
$ z_{K}(g) $ for $  g  \in G $.    
\end{definition}  \label{equivalentGktau}  One may equivalently describe $ z_{K}(g) $ as the   fundamental     cycle associated  with   the integral   $ \overline{\QQ}   $-subscheme    of  $   \Sh_{\Gb}(K) $ whose $ \CC $-points are the image of   $ X_{\Hb}^{+} \times   g    $ in $ \Sh_{\Gb}(K)(\CC) $. Then it is easily seen that  the  twisting isomorphism  $[h]_{K, h^{-1}K  h} $ for any $ h \in G $ sends $z_{K}(g) $ to $ z_{h^{-1}Kh}( g  h ) $. 

It is    clear    that $ z_{K}(g) $ are not uniquely labeled   by    $ g  \in G  $. For instance, $ z_{K}(  g  )     = z_{K}(   g    \kappa ) $ for any $  \kappa     \in K $. Let $ \mathbf{H}(\RR)_{+} $ denote the stabilizer of $ X_{\Hb}^{+} $ in $ \Hb(\RR ) $   and  set  $$ \mathbf{H}(\QQ)_{+} \colonequals \mathbf{H}(\RR)_{+} \cap  \mathbf{H}(\QQ)  .   $$    From the explicit description of $ \CC$-points underlying the scheme of  $ z_{K}(g) $, it is also clear that $ z_{K}(g) = z_{K}( h g  ) $ for any $ h \in \Hb(\QQ)_{+} $. Thus the cycles $ z_{K}(g)    $ can be  indexed by $ \Hb(\QQ)_{+} \backslash G / K $ and 
we obtain a $ \ZZ $-linear surjection 
\begin{equation}   \label{Phi} \Phi \colon \mathcal{S}_{\ZZ} ( \Hb(\QQ)_{+} \backslash G / K )  \to  \mathcal{Z}_{K}
\end{equation} 
where $ \mathcal{S}_{\ZZ}(X) $ for a set $ X $ denotes the free $ \ZZ $-module on elements of $ X $.  Since $ G /K $ is discrete, we are free to replace $ \Hb(\QQ)_{+} $ by  its  topological     closure inside     $ G $ in  (\ref{Phi}). 
If $ \Hb^{\mathrm{der}}(\RR) $ is non-compact and contained in $ \mathbf{H}(\RR)_{+}$, then this closure contains $  \Hb^{\mathrm{der}}(\Ab_{f}) $ by strong  approximation.  In this case, one obtains a large set of relations among the labels $ g $ for $ z_{K}(g) $,  even locally at each prime. 
\begin{remark} A detailed study of such cycles for certain orthogonal Shimura varieties     is done in   \cite[\S 5.14]{Cornut}.  In Proposition 5.1 of \emph{loc.cit.}, it is shown using the  Baire category theorem  that a  map closely related to $ \Phi $ is in fact a bijection.  
A  similar fact is established in \cite[\S 2.3]{Jetchev}. 
See Remark \ref{motivation} for an alternative approach to Question \ref{mainquestion} when such a description is  available.       In general, it seems unclear  what the full set of relations among the labels of these cycles are.         
\end{remark}

\subsection{RIC functors}   \label{RICfuncsec}      To address Question \ref{mainquestion} at a greater level of generality and to streamline certain arguments,  we will require some notions  of  functors on compact open subgroups of locally profinite groups  from  
\cite[\S 2]{AES}.    These notions roughly coincide with those introduced in  \cite[\S 2]{Anticyclo}, 
except that the monoid ``$ \Sigma$"   is taken to be the full group   and axiom `(T2)' of \emph{loc.cit}.\ is relaxed slightly. Note also that the terminology in the former has been slightly updated to acknowledge the prior and widespread use of these functors (and   their   variations)      in the literature (\cite{Lewis}, \cite{Bley},  \cite{BARWICK}).   We briefly recall these notions   below  and   explain  their  relevance  in  our  context.  

Fix $ \Upsilon_{G} $ any non-empty    collection of neat compact open subgroups of $ G  =   \Gb(\Ab_{f})    $ which is closed under intersections  and  conjugations by $ G $, where  any two   subgroups have the same intersection with $ \mathbf{Z}_{\Gb}(\QQ) $ and such that for any $ K, L \in \Upsilon_{G} $, there exists a third subgroup in $    \Upsilon_{G} $  that is contained in $ K \cap L $ and  normal in  $    K  $.    
It is elementary to see that such collections always exist (\cite[Lemma 2.1.1]{AES}). We let $ \mathcal{P}(G, \Upsilon_{G}) $ denote   the     small category whose  underlying set   is     $ \Upsilon_{G} $  and whose   morphisms are given by  $ \mathrm{Hom}_{\mathcal{P}(G, \Upsilon_{G})}  
( L, K ) =  \left \{ g \in G \, | \,  g^{-1} L g  \subset  K   \right \}  $ for all $L$,  $ K \in \Upsilon  $. A morphism will usually be  written as either $ (L \xrightarrow{g} K) $ or $[g]_{L,K} $ where in the latter notation,  we omit the subscripts if they are understood from context\footnote{Apriori, the notation $[g]$ here conflicts with (\ref{twistingiso}). However, $ \mathcal{P}(G, \Upsilon_{G}) $ can be canonically identified with the corresponding system of Shimura varieties and their degeneracy maps. This justifies  our    abuse of notation.}. 
Composition in $ \mathcal{P}(G, \Upsilon_{G}) $ is given by $ [g] \circ [h] = [hg] $. If $  1  =  1 _ {  G  }    $ denotes the identity of $ G $, the morphism $ [1]_{L,K} $ is also denoted by $ \pr_{L,K} $. Fix $  R $  a commutative ring  with  identity.        
   \begin{definition}    
\label{RICdefinition}      
A \emph{RIC functor $ M $ on $ (G, \Upsilon_{G} ) $ valued in    $ R $-\text{Mod}}  is a pair of covariant functors $  M ^ { * } :  \mathcal{P}(G, \Upsilon_{G}  ) ^ { \text{op} }  \to   R\text{-Mod} $ and $       M  _{  * } :  \mathcal{P}(G, \Upsilon_{G} )  \to  R\text{-Mod} $
satisfying   the   following   three   conditions:         
\begin{itemize}  \setlength\itemsep{0.3em}
\item [(C1)] $ M^{*}(K) = M_{*}(K)$ for all $ K  \in   \Upsilon   _   {  G  }   $. This common $ R $-module is denoted by $ M(K) $.
\item [(C2)] For all $ K \in \Upsilon _ { G }   $ and  $ g \in G $, $$   (gKg^{-1} \xrightarrow{g} K)^{*} =  ( K \xrightarrow{g^{-1} } gKg^{-1} )_{*}     \in  \mathrm{Hom}_{R\text{-Mod}}(M(K),M(gKg^{-1})) .  $$       
Here for a morphism $ \phi  \in   \mathcal{P} (   G   ,   \Upsilon_{G}  )  $,  we denote $ \phi_{*}  : = M_{*}(\phi) $ and $ \phi^{*}  : = M^{*}(\phi) $. 
\item [(C3)] $  [ \gamma ]_{K,K,*} : M(K) \to M(K) $ is identity map  for all $ K \in \Upsilon  _ { G   }      $ and $ \gamma \in K $.
\end{itemize}    
The pair of functors above will be denoted simply  as   $ M : \mathcal{P}(G, \Upsilon_{G} ) \to R\text{-Mod}$.   
We say  that $ M $ is    
\begin{itemize}   \setlength  \itemsep{0.1em}      
\item [(G)] $ \emph{Galois} $ if for all $ L , K \in \Upsilon  _ {  G }  $ such that  $ L \triangleleft K $,    $ \pr_{L,K}^{*}$ injects $M(K)$  onto $M(L)^{K/L}$.         
\item [(Co)]  \emph{cohomological} if for all $ L , K \in  \Upsilon _ { G }  $  with   $  L   \subset K $,  $  \pr_{L, K,   *} \circ \pr_{L, K}^{*}  = [K : L ] \cdot \mathrm{id}$.    
\item [(M)] \emph{Mackey}   
if for all $K, L, L' \in \Upsilon  _ {  G   } $ with $L ,L' \subset K$, we have a commutative diagram 
\begin{equation}   \label{Mackeydiagram}    
 \begin{tikzcd}[column   sep = large]
 \bigoplus_{\delta} M(L_{\delta}' ) \arrow[r,"{\sum [\delta]_{*}}"] & M(L) \\
 M(L') \arrow[r,"{\pr_{*}}",swap] \arrow[u,"{\bigoplus \pr_{*}}"] & M(K) \arrow[u,"{\pr^{*}}",swap] 
\end{tikzcd}
\end{equation}   
where the direct sum is over $  \delta \in  L'  \backslash K / L$ and $L_{\delta}' : =  L'  \cap  \delta L   \delta^{-1}     \in    \Upsilon_{G}        $. This  condition is independent of the choice of coset representatives.  
\end{itemize}  
When (Co) and  (M) are both satisfied, we say that $ M $ is \emph{CoMack}  for  brevity.      
\end{definition} 
\begin{remark}  The acronym ‘RIC’ stands for \emph{restriction, induction, conjugation} and may be pronounced ‘Ric’.  
To any smooth left $G$-representation $ \pi$, one can attach a RIC functor $M_{\pi} $ by setting $ M_{\pi}(K) = \pi^{K} $ (the $K $-invariants of $ \pi $)     for all $ K \in \Upsilon_{G}    $. For $ L \subset K $, the map $  \pr^{*} :   M_{\pi}(K) \to M_{\pi}(L) $ is the inclusion $ \pi^{K} \hookrightarrow  \pi^{L} $, the map    $   \pr_{*}   :     M_{\pi}(L) \to M_{\pi}(K) $ is the trace $  \sum_{\gamma \in K / L} \gamma \cdot (-) $ and        
$[g]^{*} : M(K) \to M(gKg^{-1}) $ is given by the action of $ g $ on $ \pi $.   Then $M_{\pi} $ is Galois, cohomological and Mackey.  
However, not all RIC functors arise in this manner, e.g.,  the cohomology   of     Shimura varieties with $\ZZ_{p}$-coefficients   generally  fails to be Galois.   \label{RICremark}    
\end{remark}

\begin{remark}  The role of $ \Upsilon_{G} $  in  the  above    is primarily to restrict any given  cohomology theory for  $ \Sh_{\Gb} $ to those levels where pushforwards and pullbacks behave in the expected way (i.e., the cohomology  over  varying  levels       constitutes a CoMack  functor).  This is  especially relevant in  the case of Hilbert modular varieties,  for which axiom (SV5) of  \cite{Milne} fails.    We observe that some claims in  \cite[\S 6.1]{Grossi}  do  not  hold  for  arbitrary neat 
 levels for   precisely  this reason, but this can be easily remedied by restricting to  levels in a collection $ \Upsilon_{G} $     as above  and appealing to \cite[Lemma 2.7.1]{AES}.       
\end{remark}

We   can     similarly define these notions for $ H =  \Hb(\Ab_{f}) $.  To speak  of     maps between functors on $ H $ and $ G $, we assume that the collection     
\begin{equation}    \label{upsilonH}   \Upsilon_{H} \colonequals \left \{ H \cap K \, | \, K \in \Upsilon_{\Gb}  \right   \}   
\end{equation} is such that    any two elements in $ \Upsilon_{H} $ have the same  intersection with $ \mathbf{Z}_{\Hb}(\QQ) $.    This condition is automatic if either $ \mathbf{Z}_{\Hb} \subset \mathbf{Z}_{\Gb} $ or $ (\Hb, X_{\Hb})$ satisfies (SV5), but is    otherwise a   running  assumption  in this note.      The other analogous conditions for $ \Upsilon_{H} $ however always hold.

For convenience, we will call   a pair $ (U, K ) \in \Upsilon_{H} \times \Upsilon_{G} $ \emph{compatible} if $ U $ is contained in $  K $.  
We  fix  for    the rest of this subsection    RIC functors    $   N  : \mathcal{P}(H,  \Upsilon_{H}) \to R\text{-Mod}$ and  $ M : \mathcal{P}(G,  \Upsilon_{G}) \to R\text{-Mod}  $.   

\begin{definition}   \label{pushforward}   
A  \emph{pushforward}  $  \iota_{*} : N \to M $ is a collection of   $  R  $-module   homomorphisms $ \iota_{U, K, *} : N(U) \to M(K)  $ indexed by compatible pairs $ (U, K) $  such that if $ (V, L) \in \Upsilon_{H} \times \Upsilon_{G}$ is another such pair and $ h \in H $ labels morphisms $ \phi = [h]_{V, U}  \in \mathcal{P}(H, \Upsilon_{H}) $ and $ \psi = [ h ]_{L, K}  \in \mathcal{P}(G, \Upsilon_{G}) $, we have $ \psi_{*} \circ \iota_{V, L, *} =  \iota_{U, K, *} \circ  \phi_{*}  $.
\end{definition}

\begin{definition}   \label{mixedHeckedef}        Suppose $ \iota_{*} : N \to M $ is a pushforward.
For any $ U \in \Upsilon_{H}$, $ K \in \Upsilon_{G} $ and $ \sigma \in G $, the  \emph{mixed Hecke correspondence}  $ [U \sigma K]_{*}  : N(U)  \to  M(K)    $ is defined as  the composition  $$ N (U)  \xrightarrow {  \pr ^ { * }   } N ( U \cap \sigma K \sigma ^ {  - 1 } )   \xrightarrow{\iota_{*}}  M    ( \sigma K \sigma ^{-1} ) \xrightarrow{ [\sigma]_{*} }   M  
  (K) . $$       
The  \emph{degree} of $ [U    \sigma    K]_{*} $ is defined to be the index $  [H \cap  \sigma   K  
 \sigma ^{-1} : U \cap  \sigma K   \sigma  ^{-1}  ]    $  and  denoted  by $ \deg   \,     [U  \sigma    K]_{*} $.           We say that the pushforward $ \iota_{*} $ is \emph{Mackey}  if for all $ V \in \Upsilon_{H} $ and $ L,  K \in \Upsilon_{G} $ satisfying $ V, L \subset K $, we have $\pr_{L, K}^{*}  \circ  \iota_{V, K,  *}   =  \sum \nolimits    _{\gamma}      [V \gamma L]_{*} $
where  $ \gamma $ runs over $ V \backslash K / L   $.  \end{definition}

\begin{note}   \label{Heckedef}     Suppose $ \Hb =  \Gb $ in the above situation. Then there is a  natural  pushforward $  M \to M $ given by   $ \phi_{*} $ for various    $ \phi \in \mathcal{P}(G, \Upsilon_{G}) $,  
and it is easy to see that  $ M $ is Mackey if and only if this natural pushforward is.    Given  $ K, K' \in \Upsilon_{G} $ and    $ \sigma \in G $,  we refer to $ [K \sigma K']_{*} : M(K) \to M(K') $ defined with respect this pushforward   as a \emph{covariant Hecke correspondence}  and     $ [K' \sigma K] : = [K \sigma^{-1}  K']_{*}  :  M(K) \to M(K')   $  as  a   \emph{contravariant   Hecke     correspondence}.  Then  our  notion of degree applied to these
recover the usual  notion of degrees of Hecke     correspondences. 
\end{note}

\subsection{Pushforwards of  cycles}      
\label{CoMAck} 
We now specialize the notions of \S \ref{RICfuncsec}     to the situation of interest.  Recall that for $ L \in \Upsilon_{G }$, $ \mathcal{C}^{n}(L) $ denotes the  abelian group of algebraic cycles of codimension $ n $ on $  \Sh_{\Gb}(L)   \in   \mathbf{Sch}_{\overline{\QQ}}$. Since the maps in the  inverse   system of Shimura varieties for $ \Gb   $     are  all  finite \'{e}tale, the functoriality results established   in  
\cite[\href{https://stacks.math.columbia.edu/tag/02RD}{Lemma 02RD},    \href{https://stacks.math.columbia.edu/tag/02R5}{Lemma 02R5}]{stacks-project} apply and we see that 
the groups    
   $\mathcal{C}^{n}(L) $ for varying $ L $ assemble into a RIC functor  
   \begin{equation}   \label{RICC} \mathcal{C}^{n}  \colon   \mathcal{P}(G, \Upsilon_{G}) \to  \ZZ \text{-Mod}  . 
   \end{equation}    
For $ V \in  \Upsilon_{H} $, we let $ N(V) $ denote the free abelian group on  the  fundamental  cycles  indexed  by   $ \pi_{0,  V }  $ from  \S  \ref{conn}. For the same reasons, $ N(V) $ for varying $ V $ assemble into   a RIC functor   
\begin{equation}   \label{RICN}  N  \colon   \mathcal{P}(H, \Upsilon_{H})  \to  \ZZ\text{-Mod}  .   
\end{equation}    
Moreover, there exists   a pushforward $ \mathrm{cyc}_{*}   \colon    N \to   \mathcal{C}^{n}  $ in the sense of Definition \ref{pushforward}   given by proper pushforward of algebraic cycles,  again by  \cite[\href{https://stacks.math.columbia.edu/tag/02R5}{Lemma 02R5}]{stacks-project}.      
\begin{lemma}   \label{allgoodlemma} 
The functor $ N $  and    $ \mathcal{C}^{n} $ are CoMack and Galois, and the pushforward $  \mathrm{cyc}_{*}        $ is  Mackey.       
\end{lemma}    
\begin{proof} That $ N $ and $ M $ are Mackey follows by  
 applying  \cite[\href{https://stacks.math.columbia.edu/tag/02RG}{Lemma 02RG}]{stacks-project} to the diagram in  \cite[Corollary 2.7.3]{AES}, which is Cartesian in  $ \mathbf{Sch} _{\overline{\QQ}} $ by Lemma 2.7.4 in  \emph{loc.cit}.     That $ \mathrm{cyc}  _  { *  }    $ is Mackey follows by similar considerations and our assumptions on $ \Upsilon_{H} $ (which guarantee  that the analogue of the second commutative diagram in the  proof  of  \cite[Proposition 4.12]{Anticyclo} in our context is also Cartesian).  That $ N$ and $ \mathcal{C}^{n} $ are  cohomological  follows  by    \cite[\href{https://stacks.math.columbia.edu/tag/02RH}{Lemma 02RH}]{stacks-project}.  So   it  only  remains to  see that these   functors  are    
 Galois. 
 Since $ \mathcal{C}^{n}$ specializes to $ N $  when  $  \Hb  =    \Gb $, it suffices to focus on $   \mathcal{C}^{n}  $.     So suppose that $ L , K  \in \Upsilon_{G} $ with $ L \triangleleft  K  $.  Since $ \mathcal{C}^{n} $ is cohomological, $$  \pr_{L, K} ^{* } : \mathcal{C}^{n}(K)  \to \mathcal{C}^{n} (L) $$ is necessarily injective (as its post-composition with $ \pr_{L,K}^{*} $ is multiplication by $ [K : L]$). It is also clear that the image of $ \pr_{L, K}^{*} $  lands in the $ K / L $ invariants of $ \mathcal{C}^{n}(L) $.  Finally,  the   surjectivity  of $ \pr_{L, K }^{*} $  
 follows by Galois  descent  for  algebraic  cycles  with  integral  coefficients.   Since we are unable to find a satisfactory reference for this fact, we provide a   full  proof below.       
 
Let  us  momentarily denote $  X := \Sh_{\Gb}(L)$, $ Y :=  \Sh_{\Gb} ( K) $, $ f  : X \to Y $ the degeneracy map induced by the inclusion $ L \hookrightarrow K   $ and $ \Gamma : = K / L $.  Then $ f $ is a Galois  covering  with  Galois group $  \Gamma $ in the sense of \cite[\S 6.2, Example B]{Bosch}, i.e., the induced map $ X \times \Gamma \to X \times_{Y} X $ given by $ (x, \gamma) \mapsto (x, x \gamma ) $ is an isomorphism (where 
$ \Gamma $  is      viewed    as a constant  \'{e}tale  group  scheme).    Let $ Z_{0}  $ be a closed integral  subscheme of $ X $ of  codimension $ n  $, and $ Z $ be the scheme theoretic union of the   distinct  $\Gamma$-conjugates of $ Z_{0} $.  
Then $ Z $ is a   closed and reduced subscheme of $ X $ and its  ideal  sheaf   $ \mathcal{I} $ is $ \Gamma $-invariant under the induced action $ \Gamma \times  \mathcal{O}_{X} \to  \mathcal{O}_{X} $. It 
therefore descends to a quasi-coherent sheaf of ideals  $ \mathcal{J} $ for $ \mathcal{O}_{Y} $  by \cite[\S6, Theorem 4]{Bosch}. If $ W $ denotes the  closed  subscheme of $ Y $ corresponding to $ \mathcal{J} $, we have $ f^{-1}(W) = Z $   scheme theoretically by  construction.     Then $ f^{*} $ sends  the fundamental cycle $ [W] $ to the fundamental cycle $[Z] $  by \cite[Lemma 1.7.1]{Fulton}\footnote{The scheme $ W $ can be shown to be reduced and  irreducible, but we do not need this.}. 
 Since cycles of the form  $ [Z] $  span   $    \mathcal{C}^{n}(L)^{\Gamma} $, the   surjectivity  of   $    f^{*}  =   \pr_{L, K }^{*}  $ follows.    
\end{proof}    

\begin{remark} The Galois descent for   
cycles with integral coefficients  presumably     
follows  from  
\cite{Anschutz}.   See    also  
\cite[\S 7.6.2]{Poonen}  for  an argument similar to  ours in the context of fields. For rational coefficients, Galois descent    follows by \cite[Corollary 2.1.12]{AES}
     
\end{remark}   
One may also make similar considerations when  $ \mathcal{C}^{n} $ is replaced  by    Chow groups,  or   with  the  $ p $-adic    étale cohomology with coefficients as in  \cite[\S 5.1]{Anticyclo}.  The conclusions of Lemma \ref{allgoodlemma} remain valid, except that the target functor is no longer necessarily Galois.    However, it still makes sense to   pose  the analog of Question \ref{mainquestion} in these settings. To make our results applicable to such situations, we will  work with an arbitrary CoMack functor on $ (G, \Upsilon_{G}) $  and an arbitrary  Mackey  pushforward to this CoMack  functor.    The situation of Question \ref{mainquestion} can be recovered  by  specializing  to    $ \mathcal{C}^{n} $ and $ \mathrm{cyc}_{*}     $.                

To this end, let $ R $ be any commutative ring with identity and   $  N_{R} \colonequals N \otimes_{\ZZ}  R $ be the RIC functor on $(H, \Upsilon_{H}) $ obtained by base change, i.e., $N_{R}(V) := N(V) \otimes_{\ZZ} R $ for all $ V \in \Upsilon_{H} $.  
For any $ V \in \Upsilon_{H} $, denote by $ [Y^{\circ}]  \in N_{R}(V) $ the   fundamental    cycle associated with $ Y^{\circ}   =   Y_{V}^{\circ}     $  introduced in  Definition  \ref{Ycirc}.   

\begin{definition}   \label{ykg}    Let $ M_{R} \colon  \mathcal{P}(G, \Upsilon_{G})     \to R\text{-Mod} $ be any RIC  functor and $ \iota_{*} \colon N_{R}  \to M_{R} $ be  any  Mackey pushforward. 
For $  g  \in G $ and $ K \in \Upsilon_{G} $,   we  define $ y_{K}(g ) \in  M_{R}(K) $ to be the image of $ [Y^{\circ}]    \in  N  (  H \cap g K g ^{-1}    )     $ under the composition   
\begin{equation}   \label{pushform}    N_{R}(H  \cap  g K  g  ^{-1}     )  \xrightarrow { \iota_{*} }  M_{R}  (  g  K  g  ^{-1} )  \xrightarrow {  [g ]_{*}  } M_{R}    (K)    
\end{equation} Similarly, we define $ x_{K}(g ) \in M_{R}(K) $ to be the image of the fundamental cycle $ [\mathrm{Sh}_{\Hb}(H\cap g K g ^{-1})] \in N_{R}(H \cap   g    K  g ^{-1} )  $   under   (\ref{pushform}).  
\end{definition} 
\begin{remark}   \label{xykgrel}     Since the fundamental cycle of $  \Sh_{\Hb} ( V )  $ is a formal  sum of cycles indexed by $ \pi_{0, V} $ for any $ V \in \Upsilon_{H }$, it is not hard to see that each $ x_{K}(g) $ is  a formal sum of various $ y_{K}(-) $.  Indeed, say $ Z  $ is a (geometric) connected component of $ \Sh_{\Hb}(H\cap gK g^{-1} )$ and  $ h \in H $ is such that $ Z $ is indexed by $ [\nu(h)] \in \pi_{0, H \cap g K g^{-1} }  $. Then the image of $ [Z] \in N ( H \cap g K g^{-1}  ) $ under  (\ref{pushform}) is equal to $ y_{K}(hg) $. Thus $ x_{K}(g) $ equals the sum of $ y_{K}(hg) $ as $ h $ varies over representatives for $ \pi_{0, H  \cap   g K  g^{-1 } }  $.    
\end{remark}    
As already noted,  we  consider   a   general  
ring $ R $ and  a    CoMack functor  $ M_{R} $ in order  to  capture    the  various   pushforward  constructions and cohomology theories one  may  consider.   
For instance,   one  may  take  $ R = \ZZ_{p} $ and 
$$ M_{\ZZ_{p}}(L) \colonequals  \mathrm{H}^{2n}_{\et} \big ( \mathrm{Sh}_{\Gb}(L)   
 ,  \ZZ_{p} (n)  \big     ) $$ 
 where the right hand side  denotes  Jannsen's  continuous  \'{e}tale cohomology \cite{Jannsen1988}.    
Indeed, $ N_{\ZZ_{p}}(V) $ can  be identified with $ \mathrm{H}^{0}_{\et}(\mathrm{Sh}_{\Hb}(V) , \ZZ_{p} ) $ and the pushforward is obtained via the Gysin triangle    in Ekedahl's  ``derived"      category of  constructible $ \ZZ_{p}$-sheaves as in \cite[Appendix A]{Anticyclo}.    The relevant properties of this  construction can be established  as in  \cite[\S 4]{Anticyclo}, and the failure of (SV5) can be  handled   by  \cite[Corollary 2.7.3]{AES}.  Note however that  $ M_{\ZZ_{p}} $ in this case     is   \emph{not}     necessarily  Galois.   

\begin{remark}  The inductive limit $  
\varinjlim_{L} M_{R} (L) $ for $ L \in \Upsilon_{G}  $  over restriction maps is naturally a smooth $ G$-representation, where an element  $ g \in G $ acts by $ [g]^{*} $.    The Mackey axiom for $ M_{R}   $ implies that   $$ [ K \sigma K ]_{*} \cdot x =    \sum   \nolimits    _ { \delta \in  K  \backslash  K \sigma K  }  \delta   ^ { -  1   }          \cdot x $$
for any $ x \in M_{R}    (K) $, where the  equality is  being viewed in the inductive limit.    This is the analog   of (\ref{CM2})  from the introduction. See \cite[Corollary 2.4.3]{AES} or \cite[Lemma 2.7(a)]{Anticyclo} for a justification.  
\end{remark}

\subsection{A comparison}   
Assume for this subsection only that $ M_{R }  =   \mathcal{C}^{n} $ (and $ R = \ZZ$). Then $ y_{K}( g   ) \in   \mathcal{C}^{n}(K)  $ (Definition \ref{ykg})  is    not necessarily equal to $ z_{K}( g )   $      
(Definition  \ref{zkg}),      but the two are  very    
closely related.     
\begin{lemma}   \label{dg}   There exists a unique positive integer $ d_{g, K }$ such that  $ y_{K}(g)  =  d_{g, K }   \cdot       z_{K}(g)  $.  Moreover,  if   the morphism  $ \iota _{ H \cap g  K  g  ,  g K g^{-1}  }  $  is a  closed  immersion, $ d _ {h g, K }  =  1   $ for all $ h \in H   $.    
\end{lemma} 
\begin{proof}  The first part follows by \cite[\S   1.4]{Fulton}. More  precisely,  
$      d_{g ,K} $  is the degree of the field extension of the function fields that corresponds to the dominant morphism  of  integral  schemes  
\begin{equation}  \label{dom} 
Y_{H   \cap g  K  g ^{-1}    }^{\circ}  \to \iota_{g, K} 
( Y_{ H \cap gK g ^{-1}  }  
^{\circ}  )^{\mathrm{red}}       
\end{equation}   
where $ \iota_{g, K} $ denotes $ \iota _{ H \cap g K g^{-1} , g K g^{-1} }  $ 
for simplicity and the  RHS    of (\ref{dom}) denotes the  reduced  induced  closed    subscheme  of   $  \Sh_{\Gb}( g K g^{-1}  )    $ on the image of  $Y^{\circ}$ under    $ \iota_{g ,  K  } $.             
It follows that $ d_{g , K} = 1 $ if $ \iota_{ g, K} $  is a  closed  immersion.    In  that   case,   $ \iota_{hg,  K } $ is a closed immersion for any $ h \in   H    $ as well, since we have a commutative diagram  
\begin{equation*}  
\begin{tikzcd}[column sep = large]  \Sh_{\Hb}\big(H \cap hg K (hg)^{-1} \big  )    \arrow[r, "\iota"]   \arrow[d, "{[h]}"'] &   \Sh_{\Gb} \big  ( hg K (hg)^{-1}   \big  )   \arrow[d, "{[h]}"]    \\   
\Sh_{\Hb}( H \cap  g K g^{-1} )  \arrow[r, "\iota"    ]  
&    \Sh_{\Gb} (  g K g^{-1}  )  
\end{tikzcd}  
\end{equation*}  
and vertical arrows  are    isomorphisms.    
\end{proof}   
\begin{remark} Since    the  map      (\ref{dom})  
  is   independent of the class of $  g  $ in $ \Hb(\QQ)_{+} \backslash G/ K $,    
we have $ d_{ g , K } = d_{hg \kappa      ,   K      }  $  for all $ h \in \Hb(\QQ)_{+}    $ and $ \kappa \in  K  $.    We   also observe that  the  cycle $ x_{K}(g) $ from  Definition  \ref{ykg} in this case   is closely related but not exactly the same as the  ``natural cycle" defined in \cite[\S 2]{kudla},  since the former is a sum of various $ y_{K}(-) $ (see Remark \ref{xykgrel}),  whereas the latter is a sum    of   various   $ z_{K}(-)$.     
\end{remark}

\section{The formula}     In  this  section,  we   derive our    formula for   Hecke  action  on the classes introduced in Definition \ref{ykg} and the  special cycles of Definition \ref{zkg}.  We also highlight two  scenarios where the   resulting     expression  simplifies and agrees with (\ref{falsezkg}).   
\label{theformula}

\subsection{The computation}      
The notations,  conventions  and assumptions introduced   in     \S   \ref{recollections}  and \S \ref{pushforwardsofcycles}      are maintained.  In particular, the derived group of $ \Hb $ is assumed to be simply connected and  the existence of a collection $ \Upsilon_{H} $ in (\ref{upsilonH}) whose elements have the same intersection with $ \mathbf{Z}_{\Hb}(\QQ) $ is also  assumed.   Recall  also that all our  Shimura  varieties are  viewed   in    $   \mathbf{Sch}_{\overline{\QQ}  }   $.        

As  in \S  \ref{CoMAck},  we fix   for all of this section  a commutative ring $ R $  with identity, a CoMack functor $ M_{R} : \mathcal{P}(G,  \Upsilon_{G} ) \to R \text{-Mod} $ and a Mackey  pushforward  $ \iota_{*}  :   N_{R}  \to   M_{R} $ where $N_{R} $ denotes the base change of  $ N $  (\ref{RICN}) to $ R  $.  
We also fix a compact open subgroup $ K \in  \Upsilon_{G} $ and two elements $ g,  \sigma \in G $.  Our main  goal in this subsection is to compute an expression for  $$ [K \sigma K]_{*} \cdot y_{K}(g)  \in  M_{R}(K)    $$     in  terms of $ y_{K}(-) $ from   Definition  \ref{ykg}.  Here,  $  [K \sigma K ]_{*} $ is as in  Note  \ref{Heckedef}.           
\begin{lemma}   \label{simp0}       Let $ K' \in \Upsilon_{G} $ denote  $ gKg^{-1} $.  Then $ [K \sigma K ]_{*} \cdot y_{K}(g) =  [K'  g \sigma K]_{*}  \cdot y_{K'}(1) $.    
\end{lemma}  
\begin{proof}  This follows by unraveling 
 the definitions.     
\end{proof}         
Lemma \ref{simp0}  reduces   our     problem to  computing 
$[K ' \varsigma K]_{*} \cdot y_{K'}(1) \in  M_{R}(K)   $ 
for  arbitrary $  K' \in  \Upsilon_{G}  $ and $ \varsigma \in G $, which we also fix for all of this subsection\footnote{Later on,   we will specialize to   $ K' = g K g^{-1} $ and $  \varsigma = g  \sigma  $.}.  For this purpose, we introduce the following notation.      
\begin{notation}   \label{nota1}      We let $ U $ denote  the intersection $ H \cap  K'  \in \Upsilon_{H}  $ and    $  I $  denote   the finite  double coset space $ U \backslash  K'  \varsigma K / K $.  
For each $ i \in I $, we let $ \varsigma_{i} \in G $ denote a representative element for  $   i   $.   
     
\end{notation} \begin{lemma}   \label{mixedone} 
We have  $ [K '  \varsigma    K]_{*} \cdot y_{K'}(1) =  \sum_{ i \in I }  [ U  \varsigma    _{i} K ]_{*} ( [Y_{U}^{\circ}     ] ) $ where $ [U \varsigma_{i} K ]_{*} $  denotes the mixed Hecke correspondence.       
\end{lemma} 
\begin{proof} It is elementary to deduce from the Mackey axiom for $ \iota_{*} $ that  $ [K'  \varsigma_{i} K]_{*} \circ \iota_{U,K',*}   =  \sum  \nolimits _{i \in I}  [U  \varsigma_{i} K ]_{*}   $ 
as $R$-linear maps $ N_{R}(U) \to  M_{R}(K) $  (see \cite[Lemma 2.5.7]{AES}).  The claim now  follows since $ y_{K'}(1) = \iota_{U, K', *}([Y_{U}]^{\circ}) $ by  definition.    
\end{proof}     
Lemma  \ref{mixedone} in turn reduces our problem to    computing the effect of  certain mixed Hecke  correspondences  on  $ [ Y _{U }  ^ { \circ   }    ] $.    It will be convenient to make the following notational convention.  
\begin{notation}   \label{nota2}     For any subgroup $ X $ of $  H $ and  any element  $  \tau \in G $, we denote by $ X_{\tau, K} $ the intersection $ X \cap  \tau   K  \tau ^{-1} $. We  will   write $ X_{\tau} $ for $ X_{\tau, K } $ if $ K $ is  fixed  in  context.       
\end{notation}  
Unraveling the definition  of     $[U \varsigma_{i} K]_{*}$, we  obtain  the    following commutative diagram for each $ i \in I $:         
\begin{equation} \label{mixedtwo} 
\begin{tikzcd}   & N_{R} ( U _ { \varsigma_ { i } } )    \arrow[d  ,  "{\pr_{*}}"']         \arrow[dr,  "{\iota_{*}}"]      \\ 
N_{R}(U)   \arrow[ur,   " {\pr^{*} } "  ]  \arrow[rrr,  bend right = 20,  "{[U \varsigma_{i} K ]}"]  & N_{R}( H _ { \varsigma_{i}} )   \arrow[r,  "{\iota_{*}}"]       & M_{R} (\varsigma_{i} K  \varsigma_{i}^{-1} )  \arrow[r, "{[\varsigma 
   _{i}]_{*}}"]   &     M_{R} (K)      .  
\end{tikzcd}    
\end{equation} 
We wish to compute the effect of the individual maps  in  diagram (\ref{mixedtwo}). 
To this end,  recall (\S \ref{conn})      that for $ V \in \Upsilon_{H} $, $ \pi_{0, V} $ is an abelian  group  that   parametrizes   the   connected components of  $ \mathrm{Sh}_{\Hb}(V) $ and for $ V '   \in     \Upsilon_{H} $ contained in $ V $,   $ \varphi_{V', V} : \pi_{0, V} \to \pi_{0  ,  V '  } $ denotes       the quotient morphism. 
\begin{notation}   \label{nota3}     For any   $ \tau   \in  G $,  we   let     $ A_{\tau} $  denote a set of representatives in $ H $   for the kernel of the  homomorphism   $ \varphi_{U_{\tau}   ,   U   }   :  \pi_{0,  U_{\tau}}   \to   \pi_{0, U }   $.  
\end{notation}   
Here the representatives are picked under the composition $ H   \xrightarrow{\nu} \Tb(\Ab_{f})  \to  \pi_{0, U_{\tau}} $, which is surjective  by \cite[Lemma 5.21]{Milne} and simply  connectedness of $ \mathbf{H}^{\mathrm{der}} $. 
\begin{lemma} For any $  \tau  \in G $,  
\label{mixedthree}     
$$\pr^{*} _{  U _ { \tau  } ,  U } \big ( [ Y_{U} ^ { \circ}    ] )    =  \sum   \nolimits     _ {  h  \in  A_{   \tau   }   } [h]_{*}  \big ( [ Y    _   {      h  U_{  \tau   
}    h ^ { - 1 }  }   ^{\circ}    ]   \big )   \in  N_{R}    (U_{  \tau    } ) $$
where $ [h] $ on the  right  hand  side  
above denotes the morphism $  ( h U _{ \tau } h ^{-1}  \xrightarrow{h}  U_{\tau } )     \in    \mathcal{P}(H ,   \Upsilon_{H}  )     $.    
\end{lemma} 
\begin{proof}  Let $ \eta \in H $.  Recall (\S \ref{conn})         that the twisting  isomorphism    $     [h] \colon \mathrm{Sh}_{\Hb}( h U_{\tau }  h ^{-1} )  \xrightarrow { \sim }  \mathrm{Sh}_{\Hb}( U  _ {  \tau    }    ) $  sends  the  component indexed by the class of  $ \eta  $ in $  \pi_{0 , h U_{\varsigma} h^{-1} }  =  \pi_{0,  U_{\varsigma} }  $   to  the   component of $ \Sh_{\Hb}(U_{\varsigma}) $ indexed by    the   
    class of $ h \eta  $ (or that    of     $ \eta h $)  in  $ \pi_{0,  U_{\varsigma  }    }      $.   So the RHS of the equality above  is  just the formal sum of the  fundamental cycles  
corresponding to  the    connected components of $ \mathrm{Sh}_{\Hb}(U_{\tau}  ) $ indexed by $ A_{\tau }   $.    We argue that this also equals the LHS.   Since  the  morphism $ \pr_{U_{\tau }, U} \colon \mathrm{Sh}_{\Hb}(U_{\tau }) \to  \mathrm{Sh}_{\Hb}(U) $ is  \'{e}tale, so  is its pullback $ \pr_{U_{\tau}, U } ^{-1} ( Y_{U} ^ {\circ} ) \to  Y_{U}^{\circ} $ along $ Y_{U}^{\circ}  \to \mathrm{Sh}_{\Hb}(U) $. This  implies that  $ \pr_{U_{\tau}, U}^{-1}(Y_{U}^{\circ}) $ is reduced 
\cite[\href{https://stacks.math.columbia.edu/tag/03PC}{Tag 03PC}]{stacks-project}). So the scheme  $  \pr^{-1}_{U_{\tau}, U}(Y_{U}^{\circ}) $ is equal to the disjoint union of  the  components of $ \Sh_{\Hb}(U_{\tau}) $ indexed by $ A_{\tau}$. The  claim  now   follows by definition of flat pullback  \cite[\S 1.7,  \S 1.5]{Fulton}. \end{proof}    
 
Next we need  a  result for degrees of maps  between  connected  components of    $ \mathrm{Sh}_{\Hb} $. 
\begin{lemma}   \label{mixedfour}   Let $ V_{1}, V_{2} \in \Upsilon _ { H }  $ with $ V_{2} \subset V_{1} $ and $ \eta \in \Tb(\Ab_{f}) $.  For $  j  = 1, 2 $,  let $ Y_{\eta, V_{j}} $ denote the component of $ \mathrm{Sh}_{\Hb}(V_{j}) $  indexed by  the class of $ \eta $ in $  \pi_{0,  V_{j}} $ and let  $ e $ denote the cardinality of the kernel     of    $ \varphi_{V_{2}, V_{1}} \colon   \pi_{0, V_{2}}    \to   \pi_{0,  V_{1} } $.    Then the  natural map $ Y_{\eta, V_{2}}  \to  Y_{\eta,  V_{1}} $ is finite \'{e}tale of degree $ [V_{1} :  V_{2}]/e $.         
\end{lemma} 
\begin{proof}   Let $ W $ be a compact open subgroup of $ V_{2}  $   such that $  W   \triangleleft V_{1} $. To ease notation, we let $ V $ denote an element of $  \left  \{ V_{1} ,V_{2} \right \} $. By enlarging $ W $ if necessary, we may assume that  $  W \cap \mathbf{Z}_{\Hb}(\QQ) = V \cap \mathbf{Z}_{\Hb}(\QQ) $, so that  the map $ \mathrm{Sh}_{\Hb}(W) \to \mathrm{Sh}_{\Hb}(V) $ is a Galois cover with Galois group $ V /W $. Since $ Y_{\eta, V} \hookrightarrow \Sh_{\Hb}(V)    
  $ is an open immersion, $$   Z \colonequals \pr^{-1}_{W,V}(Y_{\eta,V   }  )  \to  Y_{\eta,V} $$ is a Galois cover  of  degree of $ [V : W] $ as well. Let $ e_{V} $ denote  the 
 cardinality  for   $ \ker  ( \varphi_{W, V}   )      $ and let $ \upsilon_{1}, \ldots \upsilon_{e_{V}}  \in V $ be representatives of $ \ker (   \varphi_{W, V}    )    $. Then  $  Z  $    is the   union of components of $ \Sh_{\Hb}(W)    
 $     indexed by the classes $ [\eta \upsilon_{k}] \in  \pi_{0,   W   }    $ for $ k =1, \ldots, e_{V} $.       Since $ [\upsilon_{k} ]    \colon \Sh_{\Hb}(W)  \to \Sh_{\Hb}(W) $ are automorphisms that  act transitively on the   connected components  contained  in $ Z $,  we see that the degree of $ Y_{\eta \upsilon_{k}, W } \to Y_{\eta, V} $ is independent of $ k $ and therefore equal to $ [V : W  ]/e_{V}     $. 
 Since $ e_{V_{1}} = e_{V_{2}} \cdot  e $ and $ [V_{1} : W]  =  [V_{1} : V_{2}]   \cdot [V_{2} : W] $, we conclude that  $ Y_{\eta, V_{2}} \to Y_{\eta, V_{1}} $ is finite \'{e}tale  with degree  as  claimed.    
\end{proof} 
\begin{remark} Note that in the proof, we do not require $ W \in \Upsilon_{H} $.   
\end{remark}   
\begin{corollary}   \label{mixedfour'}  For any $ \tau \in G $ and $ \eta  \in  \pi_{0,  U_{\tau}} $, $$ \pr_{U_{\tau}, H_{\tau}, *}(  [ Y_{\eta,  U_{\tau}}]  )   =      e_{\tau}^{-1}   \deg \, [ U  \tau K ]_{*}   \cdot [Y_{\eta,  H_{\tau}}]  $$ 
where $ Y_{\eta, U_{\tau}}$ denotes the  component of $ \Sh_{\Hb}(U_{\tau}) $ indexed by $ \eta $, $ Y_{\eta, H_{\tau}} $  denotes  the component of $  \Sh_{\Hb}(H_{\tau}) $ indexed by $ \varphi_{U_{\tau}, H_{\tau}}(\eta)  \in   \pi_{0 ,   H_{\tau}}  $ and $ e_{\tau} $   denotes the cardinality of $ \ker (\varphi_{U_{\tau},  H_{\tau}} ) $.          
\end{corollary}   
\begin{proof}  This  follows by Lemma \ref{mixedfour} and   the   definition  of  proper   pushforward  \cite[\S 1.4]{Fulton}.        
\end{proof}    
We now apply these results to the maps in diagram (\ref{mixedtwo})  by  specializing to $ \tau =  \varsigma_{i}   $.    Let us first fix some  additional  notation.    
\begin{notation}   \label{nota4}     For each $ i \in I $, we will let $U_{i} $, $H_{i}$,  $ A_{i} $  denote $ U_{\varsigma_{i}} , H_{\varsigma_{i}} , A_{\varsigma_{i}} $   respectively.    We let   $ B_{i}  \subset  A_{i}    $ denote   a    set of representatives for the image of $ A_{i }$ under the quotient morphism  $  \varphi _ { U _{ i }  ,  H_{i}   }      :   \pi_{0,  U_{i}} \to   \pi_{0,   H_{i }    }    $.    Here  we   are identifying $ A_{i} = A_{\varsigma_{i}}   $ with kernel of $  \varphi_{ U_{i},  U}   :   \pi_{0,  U_{i}}  \to    \pi_{0,  U}     $.   We let $ e_{i}  =   e_{\varsigma_{i}}     $  denote the cardinality of $ \ker  (    \varphi_{U_{i} , H_{i}   }  ) $ and $ c_{i}  $ denote  the cardinality of $ \ker ( \varphi_{U_{i},  H_{i} }   )      \cap   \ker    (   \varphi_{U_{i}, U}) $. 
\end{notation}    
If we  identify $ A_{i} $ with $ \ker (\varphi_{U_{i}, U}) $ and $ B_{i} $ with $ \varphi_{U_{i}, H_{i}}(A_{i}) $, then $ c_{i }$ is the kernel of the surjective  homomorphism    $ A_{i}  \to  B_{i}   $, and so  $ c_{i} = |A_{i}|/ |B_{i}| $. 
Let us  also   emphasize that $ H_{i} = H_{\varsigma_{i}} = H \cap  \varsigma_{i} K \varsigma_{i}^{-1} $ is a \emph{compact open} subgroup of $ H $ (despite the notational similarity  with  $  H $). To make the statement of our main   theorem more self-contained, we recall  most  of     the necessary notations. 
\begin{theorem}[Explicit descent] \label{descent}    
For  $ K' , K \in  \Upsilon_{G}    $ and  $ \varsigma \in G  $, denote  $ U =  H \cap  K '    $ and $  I = U \backslash K  '    \varsigma K / K  $. For each $ i \in I $, let $ \varsigma_{i} \in G $ denote a representative  for the class  $ i $ and  denote $  U_{i}  =   U  \cap  \varsigma_{i} K  \varsigma_{i} ^{-1} $, $ H_{i} = H  \cap  \varsigma_{i}  K  \varsigma_{i} ^{-1} $ and  $ \deg \, [U \varsigma_{i} K ]_{*} = [  H_{i} : U_{i} ] $.  
For each $ i $, let $ A_{i} \subset H $ denote a set of representatives for $  \ker(\varphi_{U_{i}, U }  ) $,  $ B_{i} \subset A_{i} $ denote a set of representatives for $ \varphi_{U_{i},  H_{i}}( \ker ( \varphi_{U,   U_{i} } )  ) $ and  set   $ c_{i}   =     |A_{i} |  /  |B_{i}|  $, $ e_{i} = | \ker (\varphi_{U_{i}, H_{i}}  )| $.      Then      
$$  [K   '  \varsigma K ]_{*}  \cdot  y_{K'}(1)   =   \sum_{ i \in I }  \sum  _ {   h    \in  B_{i} }   c_{i}  e_{i}^{-1} \deg [U   \varsigma_{i}   K ]_{*}    \cdot     y_{K}(   h  \varsigma _{i} )  $$ 
as elements of $ M_{R}(K)  $. 
\end{theorem}   
\begin{proof}   By Lemma  \ref{mixedone}, it suffices to compute $ [U   \varsigma _{i} K]_{*} ( [Y_{U}]^{\circ}) $ for  each  $ i \in I $.   Note that the integer  $  c_{i} $ is the number of connected  components of $ \Sh_{\Hb}(U_{i}) $  contained  in  
 $ \pr_{U_{i}, U} ^{-1}(Y_{U}^{\circ}) $ that collapse into  a  single  component  of  $ \Sh_{\Hb}(H_{i}) $       under  $  \pr_{U_{i},  H_{i} }  $.     
Invoking Lemma \ref{mixedthree}  and  Corollary  \ref{mixedfour'}, we see that 
\begin{align}  
\pr_{ U_{  i },  H_{   i    }     , *}     \circ   \pr^{*} _ {  U _  { i }  ,  U  }     \big  (    [ Y _{ U}  ^ { \circ}  ]    \big     ) 
&  =     \sum   \nolimits      _ {   h    \in B_{i} }  \big  ( c_{i}  e _ { i } ^ { -1 }  
\deg[U   \varsigma_{i}    K ]_{*}  \big )   \cdot      [ h  ]_{*}     \big (  [Y^{\circ}_{H_{h \varsigma_{i}}} ]   \big  )       .    \label{mixedfive}       
\end{align}
where $ [h] $ above is the morphism $ (H_{h\varsigma_{i}}  \xrightarrow{h}H_{\varsigma_{i}} )\in   \mathcal{P}(H,\Upsilon_{H})$.  Now for each $ i \in I  $  and   $ h  \in B_{i} $,  we  have  a  commutative  diagram 
\begin{equation}    \label{mixedcomm}        
\begin{tikzcd}[sep =  large]     N_{R}( H_{ h  \varsigma  _{i}} ) \arrow[r  ,   "{[\iota]_{*}}"]   \arrow[d,   "{[h]_{*}}",   "{\resizebox{0.7cm}{0.08cm}{$\sim$}}" labl1]  &    M_{R} (  h \varsigma_{i} K \varsigma_{i}^{-1} h^{-1}   )    \arrow[d, "{[h]_{*}}",  "{\resizebox{0.7cm}{0.08cm}{$\sim$}}"' labl1]    \arrow[dr,   "{[  h    \varsigma _ { i } ]_{*}}" ]         \\
N_{R} (H_{   \varsigma    _{i}} )    \arrow[r  ,  "{[\iota]_{*}}"]    &   M_{R}     (  \varsigma_{i}  K  \varsigma_{i}^ {  -  1   }   )   \arrow[r  ,   "{[\varsigma_{i}]_{*}}"]   &   M_{R}    (K)  
\end{tikzcd} 
\end{equation}     
Let us  momentarily  denote $ K_{i} = \varsigma_{i} K \varsigma_{i}^{-1} $ to  simplify  notation.   Using    (\ref{mixedfive})     and the commutativity of diagrams (\ref{mixedtwo}) and (\ref{mixedcomm}),   
  we  see that 
\begin{align*}   
[U \varsigma_{i}K]_{*}\big([Y_{U}^{\circ}]\big)  & = [\varsigma_{i}]_{K_{i}, K,*} \circ \iota_{U_{i}, K_{i},*}\circ \pr^{*}_{U_{i},U}\big([Y_{U}^{\circ}]\big)\\
& = [\varsigma_{i}]_{K_{i},K,*}\circ \iota_{H_{i}, K_{i},*}\circ\pr_{U_{i},H_{i},*}  \circ \pr^{*}_{U_{i},U}\big([Y_{U}^{\circ}]\big)\\    
& = [\varsigma_{i}]_{K_{i},K,*}\circ\iota_{H_{i},K_{i},*} \bigg(\sum\nolimits_{h\in B_{i}}\big(c_{i}e_{i}^{-1}\deg[U \varsigma_{i} K ]_{*}\big)\cdot[h]_{H_{h\varsigma_{i}} ,  H_{i} , *}\big([Y_{H_{h\varsigma_{i}} }]\big)\bigg )  
 \\   
& = \sum\nolimits_{h \in B_{i}}\big(c_{i}e_{i}^{-1}\deg[U\varsigma_{i} K]_{*}\big)\cdot[h\varsigma_{i}]_{ hK_{i}h^{-1} , K , *}\circ\iota_{H_{h\varsigma_{i}} , hK_{i}h^ { -1},*}\big([Y_{H_{h\varsigma_{i}}}^{\circ}]\big)
\\
& = \sum \nolimits _ { h \in B_{i}}   c    _  { i  }      e_{i}^{-1} \deg [U \varsigma_{i} K ]_{*}  \cdot      y_{K} ( h   \varsigma _ { i } )   
\end{align*}   
which  finishes  the   proof.  
\end{proof}  
In the formula above, we may require the inner sum to be over $ A_{i} $ (instead of    $ B_{i} $) after  removing $ c_{i}   $  from  the  expression, since $ y_{K}(h\varsigma_{i}) = y_{K}(h' \varsigma_{i}) $ for $ h , h' \in A_{i} $ if the classes of $ h , h ' $ are equal in  $ \pi_{0,   H_{i} } $.

\begin{theorembis}{descent}  With notations  as above,  $$ [K   '  \varsigma K ]_{*}  \cdot  y_{K'}(1)   =   \sum_{ i \in I }  \sum  _ {   h    \in  A_{i} }    e_{i}^{-1} \deg [U   \varsigma_{i}   K ]_{*}    \cdot     y_{K}(   h  \varsigma _{i} )  . $$
\end{theorembis}    

One reason for preferring the first version is that a simplification occurs when    $ \nu(U)  $ contains $ \nu(H_{i}    ) $  for all  $   i  \in  I   $.   We record it as a lemma for ease of reference in \S \ref{mainfalseexamplessec}.    
\begin{lemma}    \label{Bilemma}  If $ \nu(U) $  contains $ \nu (H_{i}) $ for some $ i \in I $, we have $ c_{i} = e_{i} $. If moreover $ \nu(U) $ equals $ \nu(H_{i}) $, $ B_{i} $ is a singleton.    
\end{lemma} 
\begin{remark} That $ \nu(U) \supset \nu(H_{i}   
  ) $ for all $ i \in I $ holds, for instance,   if $ \sigma \in \Gb(\QQ_{\ell}) $ for some rational prime $ \ell $  where 
$ \Tb $ is unramified, $ U $ is of the form $ U_{\ell} U^{\ell} $ for $  U_{\ell} \subset \Hb(\QQ_{\ell}) $ and $ \nu(U_{\ell})  \subset  \Tb(\QQ_{\ell}) $ is the unique maximal  compact  open   subgroup. So in this case,  
the coefficients in the expression  of   Theorem \ref{descent}  only involve   mixed  degrees. 
See  \cite[\S 5]{AES} where  several  techniques were developed to aid     their computation and Part II of \emph{op.cit.} for several   concrete     examples. 
\end{remark}

If we replace $ y_{K'}(1) $ with $ x_{K'}(1)$,  the formula is much simpler and does not require as much work.  
\begin{proposition}  \label{xKgdescent}  We  have  $  [K'  \varsigma   K ]_{*}  \cdot x_{K'}(1)   =   \sum_{ i \in I }   \deg   \, [U \varsigma_{i} K ]_{*}    \cdot     x_{K}(  \varsigma_{i} )  $.  
\end{proposition}

\begin{proof}  Let $ N_{\mathrm{triv}, R} : \mathcal{P}(H,  \Upsilon_{H})  \to R\text{-Mod} $ denote functor associated with the trivial representation of $H  $ (see Remark  \ref{RICremark}), so that  $ N_{\mathrm{triv}, R}(V) = R $ for all $ V $.  Since  fundamental  cycles of  $ \mathrm{Sh}_{\Hb}(V)$ for varying $  V \in \Upsilon_{H} $  map to  themselves under pullbacks and to multiples by degree under pushforwards 
along degeneracy maps, they realize the trivial functor $ N_{\mathrm{triv},R} $. The class $ x_{K}(g) $ can then be defined as the image of $ 1_{R} \in R = N_{\mathrm{triv}}(H_{g}) $ under the analogous twisting map in Definition \ref{ykg}. The claim easily follows by the obvious analog of Lemma \ref{mixedone} and the  diagram (\ref{mixedtwo}) with $ N_{R} $ replaced by $ N_{\mathrm{triv}, R} $.   
\end{proof}

\begin{remark} Proposition \ref{xKgdescent}   holds  without the assumption  that $ \mathbf{H}^{\mathrm{der}}$ be simply-connected, since the definition of $ x_{K}(-)$ etc.,  does not     rely on a description of the connected components of $ \Sh_{\Hb} $.   One may also drop (SV3) for $ (\Hb, X_{\Hb})$  in light of \cite[Appendix A]{Anticyclo}, which extends the formalism of Shimura varieties in the absence of (SV3), \emph{assuming} that  $ (\Hb, X_{\Hb}) $ embeds into a data which does satisfy (SV3).  In our case, this latter data is $ (\Gb, X_{\Gb}) $.            
\end{remark}

If we specialize  $ K'  = g  K  g^{-1} $ and $ \varsigma = g \sigma  $ in Theorem  \ref{descent} and invoke Lemma  \ref{simp0},     we  obtain 
the  following.

\begin{corollary}   \label{ykgcoro} For $ K \in \Upsilon_{G} $ and $ g ,   \sigma \in G $, denote $ U = H \cap  g K  g^{-1}    $ and $ I = U  \backslash g K \sigma K / K $.   For each  $  i  \in  I $,  let $ \varsigma _ { i } \in  G $ denote a representative for $ i   $ and denote $ U_{i} = U \cap \varsigma_{i} K \varsigma_{i}^{-1} $, $ H_{i} = H \cap  \varsigma_{i}  K   \varsigma_{i} ^{-1}  $ and  $ \deg [U \varsigma_{i} K ] _{*} =  [H_{i} : U_{i}] $. For each $ i $, let $ A_{i} \subset H $ denote a set of representatives for $ \ker ( \varphi_{U_{i}, U}   )     $, $ B_{i} \subset  A_{i} $ denote a set of representatives for $ \varphi_{U_{i}, H_{i}}( \ker ( \varphi_{U , U_{i}})   ) $ and set $ c_{i}  =   |A_{i}| /   |B_{i}  |  $,   $ e_{i} = | \ker (\varphi_{U_{i}, H_{i}}  )| $.    Then  $$ [ K \sigma  K ]_{*}  \cdot  y_{K}(g)  =  \sum_{i \in I  }  \sum_{ h \in B_{i}  }  c_{i  }    \,   e_{i}  ^{-1}      \deg \, [U  \varsigma_{i}  K ]_{*}  \cdot  y_{K} ( h  \varsigma_{ i } )   . $$
\end{corollary}

We    can finally  answer Question \ref{mainquestion} now. Recall from Lemma  \ref{dg} that when $ M_{R} = \mathcal{C}^{n} $ (and $ R = \ZZ$), there  exists for each $ \tau \in G $  a  unique positive integer $ d_{\tau,  K } $  such that  $ y_{K}(\tau)  =  d_{\tau, K }   \cdot     z_{K}(\tau) $ as elements of $ \mathcal{C}^{n}(K) $. 
\begin{corollary}   \label{zkgcoro}  With notations as  in  Corollary  \ref{ykgcoro},   we  have   
\begin{equation}     \label{zkgcoroexp}     [K \sigma  K ]_{*}   \cdot  z_{K}( g )   =   \frac{1}{d_{g,  K }         }         \sum_{  i  \in    I}  \sum _{   h \in   B_{i}    }       c_{i}  \,  d_{ h \varsigma_{i} , K } \,   e_{i}^{-1}      \deg  \,  [ U  \varsigma _{i}    K  ]_{*}  \cdot   z_{K}(   h    \varsigma _{i} )   
\end{equation}    
In particular,   $    [ K  \sigma  K ]_{*} \cdot z_{K}(g) $ lies in the submodule of $ \mathcal{C}^{n}(K) $ spanned by irreducible special  cycles.   
\end{corollary}

\begin{proof}   The   expression follows  from  Corollary  \ref{ykgcoro} after specializing to $ R  =   \ZZ   $,  $ M_{R} =  \mathcal{C}^{n} $ and $   \iota _  { *  }  =  \mathrm{cyc}_{*}  $ (which we can do  by   
Lemma  \ref{allgoodlemma}).   
Since $ [ K \sigma  K ]_{*} \cdot z_{K}(g) $ belongs to both $ \mathcal{C}^{n}(K) $ (by definition) and  $ \mathcal{Z}_{K} \otimes_{\ZZ}  \QQ $ by (\ref{zkgcoroexp}),  it must lie    
inside   $ \mathcal{Z}_{K} $ (where $   \mathcal{Z}_{K}  $ is as  in   Definition      \ref{zkg}).     
\end{proof}    
\begin{remark}   Note that the  coefficients of individual summands 
in (\ref{zkgcoroexp}) are not  apriori integers, since we do not know if  $ d_{g, K} $ divides the  product   $c_{i}  \,  d_{ h \varsigma_{i}}  \,   e_{i}^ {  -  1   }       \deg [U  \varsigma_{i} K]_{*} $ for all $  i \in I $,    $  h    \in  B_{i}    $.   The point is that the    coefficients   in our  expression     can be made integral (if not already)  by collecting  together  the coefficients of all $ z_{K}(h \varsigma_{i}) $ that represent  the  same  irreducible cycle in $ \Sh_{\Gb}(K)  $.  
\end{remark}     

\subsection{Examples} 
Below,  we  record two simple   instances  in which  the  RHS of (\ref{zkgcoroexp})  matches that of    (\ref{falsezkg}).  The notations above  are  maintained.       
\begin{example}   \label{mixedCM}        Suppose $ \Hb    $ is a torus. In this case, we are asking for Hecke action on special points on $ \Sh_{\Gb}(K) $.  
We have $ \Hb =  \Tb $ and $ \nu  $ is  the   identity     map.    Our assumption on $ \Upsilon_{H} $ imply that $ A_{i} $ identifies with $ 
U / U_{i}  $ and $ e_{i}   
 = [H_{i}: U_{i} ]   =   \deg \, [U \varsigma_{i} K ]_{*}  $.       Moreover,  have $ d_{\tau, K}   = 1 $ for all $  \tau   \in G    $,     since  $ \Sh_{\Hb}(H_{\tau}) $ is a finite set of   reduced     points over $ \overline{\QQ} $.  
Thus \begin{align*}  [K  \sigma  K]_{*}   \cdot  z_{K}(g)   &   =    \sum_{i \in  I }  \sum _{   h   \in  B_{i}    }  c_{i} 
\cdot   z_{K}( h  \varsigma _{i} )  = \sum _ { i \in I } \sum _ { h  \in A_{i} } z_{K} (  h  
 \varsigma _{i} ) 
\end{align*} 
Now for each $ i \in I $, we have $ U \varsigma_{i} K / K = \bigsqcup _ { h 
 \in  A_{i} } h \varsigma _{i} K $  and therefore  
$ g K   \sigma    K / K = \bigsqcup_ { i \in I, h \in A_{i } } h \varsigma_{i } K $. So   
$$  [K  \sigma    K ] _{ *}     \cdot     z_{K}( g ) =  \sum   \nolimits    _ { \gamma \in K   \sigma     K / K }  
z_{K}( g \gamma )   .  $$
which agrees with   (\ref{falsezkg}).    This  is  of course what one gets by  directly computing the result of a  correspondence on a general zero-cycle on $  \Sh_{\Gb}(K)  $  as in   (\ref{CM1}).     
\end{example}  
\begin{example}  Suppose  that  $ \mathbf{H} = \Gb $.  In  this  case,   we  are asking for Hecke action on connected components of $ \Sh_{\Gb}(K) $  itself.    We   have    $    U    = g K  g    ^{-1}  $, so that $ I = \left \{ 1  \right \} $ is  a  singleton and we may take   $ \varsigma_{1} = g   \sigma     $. By definition,  we  have   $   \deg \,  [ U  \varsigma _{1} K]  _ { *  }   =  [ \varsigma_{1} K \varsigma_{1}^{-1} :  U_{\varsigma_{1}   }  ] $,  which equals $  [K : \varsigma_{1}^{-1} U_{\varsigma_{1}} \varsigma_{1} ]$.  Since $ \varsigma_{1}  ^{-1}  U   _  {  \varsigma  _  {  1  }    }    \varsigma_{1}^{-1} =  \sigma^{-1} K \sigma \cap K $,  we see that  $$ \deg \, [  U  \varsigma_{1}   K  ] _{*}  =   |  K  \sigma^{-1} K  / K  |   .          $$  
Since $ \nu(  U ) = \nu(K)  =  \nu  ( H_{\varsigma_{1}   } ) $,  we  have  
$ c_{1} =  e_{1}$ and $ B_{1} = \left \{ 1 \right \} $ 
a singleton by
 Lemma    \ref{Bilemma}.  Clearly,  $  d_{\tau, K } = 1 $ for all $  \tau   \in G $.  Therefore $$ [K  \sigma  K ]_{*} \cdot z_{K}(g)  =  | K  \sigma^{-1}  K / K  |   \cdot  z_{K}( g  \sigma   )  .  $$
Now $ | K  \sigma ^ { - 1 }  K / K |  =  | K   \sigma    K / K | $ by unimodularity  of $ G $ \cite[p.\ 58]{Renard}   
and $ z_{K}(g \gamma \sigma ) = z_{K}( g   \sigma  ) $ for any $ \gamma \in K $ since $ \pi_{0, K } =  \pi_{0}  (   \Sh_{\Gb}(K)  )  $ is  abelian.      So   $$  [K  \sigma K ]_{*} \cdot  z_{K}(g)  =  \sum _ { \gamma \in K  \sigma    K / K } z_{K}( g \gamma)   ,   $$   which  again  agrees    with (\ref{falsezkg}). 
   
\end{example}

\section{Counterexamples}    
In this section, we furnish two  (families of)    examples   where      (\ref{falsezkg})  fails  to     hold. 
\label{mainfalseexamplessec}
\subsection{Counting cycles} 

Since irreducible special cycles form a $ \ZZ $-module  basis of $ \mathcal{Z}_{K} $ by definition,     we see that   (\ref{falsezkg})    holds   only  if    $$  | K \sigma K / K |   \stackrel{?}{=}    
   \frac{1}{d_{g, K }   }    \sum _ {  i   \in   I }   \sum_{ h \in B_{i} }    c_{i} \,  d_{h \varsigma_{i} ,   K    }\, e_{i}^{-1} 
 \deg   \,     [ U  \varsigma_{i}   K  ]_{*}  . $$ 
 Indeed, the RHS  above is the number of basis elements in the RHS of  (\ref{zkgcoroexp}).   The strategy is therefore  to compute both these integers explicitly and show they are not equal.   We will pick $ K$  and $ \sigma $ so that the various $ d_{-, K 
 }    $ appearing above are forced to be $ 1 $.    For the computation of mixed degrees, we rely on the calculations done in  \cite{AES}, though   alternatives are also provided  for the reader and the computations are mostly  self-contained.          In both our examples, we will have $ g = 1_{G}  $ (so $ U = H \cap K )$ and the letter $ g $ will be used for  other purposes.  For this reason, we denote the representatives  $ \varsigma_{i} $ for $   i  \in   U  \backslash  K  \sigma   K   /   K   $ by $ \sigma_{i} $.   In this notation, we  are interested in checking  if       
\begin{equation}   \label{mainfalse}       |K \sigma K / K |   \stackrel{?}{=}   \frac{1}{d_{1_{G}, K } }   
\sum _{i \in I} \sum_{h \in B_{i}}  c_{i} \,   d_{h  \sigma_{i} ,  K }  \,  e_{i}^{-1} 
\deg [U   \sigma    _{i} K ]_{*}. 
\end{equation}   
Throughout \S \ref{mainfalseexampleone} and \S \ref{mainfalseexampletwo},  $ \ell $ denotes a rational  prime and $ \mathbb{A}_{f}^{\ell} =  \mathbb{A}_{f}/\QQ_{\ell}  $ denotes    the  group   of   finite        rational  adeles  away from $  \ell    $.  
\begin{remark} The counterexamples below also work for if $ z_{K}(1) $ is replaced by  $ x_{K}(1) $.    
\end{remark}

\subsection{Symplectic groups}    \label{mainfalseexampleone}  We  let    $$  \Hb = \GL_{2} \times _{\GG_{m}}  \GL_{2} , \quad \quad  \Gb = \mathrm{GSp}_{4}  ,  $$
which we consider as reductive group schemes  over  $ \ZZ $. 
Here, we define  $ \Gb $   with respect to the standard  symplectic   matrix which has the identity matrix in the top right $ 2 \times 2 $ block, negative identity in the bottom left $ 2 \times 2 $ block  and zeros elsewhere.   The embedding of $ \Hb $ in $ \Gb $ is    as  in   \cite[\S 9.3]{AES}, which gives a morphism of Shimura data (see \cite[\S 6]{Milne}). Moreover,  both data satisfy (SV1)-(SV6) of \cite{Milne}  by  the  discussion  in  \S 6 of  \emph{op.cit}.  In  addition,    the derived group of $ \Hb $ is $ \SL_{2} \times \SL_{2} $,  which  is  simply  connected. Finally, our assumption on the existence of $ \Upsilon_{H} $ is also  satisfied,  since  the data for $ \Hb $  satisfies  (SV5).    Thus Corollary  \ref{zkgcoro}  applies to this embedding of Shimura data. Let  us   denote  $$  w \colonequals \mathrm{diag}(1,-1,1,-1)  \in \mathbf{Z}_{\Hb}(\QQ)  .   $$     
Then the centralizer of $ w $ in $ \Gb $  equals   $    \Hb $. We have $ \Tb = \Hb/ \Hb^{\mathrm{der}} =   \mathbb{G}      _{m} $ and $ \nu \colon \Hb \to \Tb $ is  the  map given  by  taking  the common determinant of the two components.   As above, we denote  $ G = 
     \Gb(\Ab_{f}) $ and $ H  =  \Hb ( \Ab_{f} )  $. 
For $ N \geq 1 $,  we let $ K(N) $ denote the  principal congruence subgroup of $ G $  level $ N $. Then  $ K(N) $ is neat if $ N \geq 3 $. Now say $ K $ is a  compact open subgroup contained in $ K(N) $ for some $ N \geq 3 $.  
Since $ w \in K(1) =   \Gb    (\widehat{\ZZ}) $ and $ K(N) \trianglelefteq K(1) $, we have $ w K w ^{-1} \subset w K(N) w^{-1}  = K (N) $. Thus $$ \iota_{U,K}  \colon \Sh_{\Hb}(U)  \to   \Sh_{\Gb}(K)          $$    is a closed immersion  by Lemma  \ref{closedimmersion}. We fix such a  $ K $ from now on and denote as above $ U  \colonequals H \cap K $, etc.   We let  $ \ell $  be  a  rational  prime  such that $ K $ is unramified at $ \ell $, i.e.,    $ K = K_{\ell }  K^{\ell} $ where $ K_{\ell} =  \mathrm{GSp}_{4} ( \ZZ_{\ell}  )  $   and    $ K^{\ell}  \subset \mathrm{GSp}_{4}(\Ab_{f}^{\ell}) $ is a compact open subgroup.    Then $ U = U_{\ell} U^{\ell} $ with $U_{\ell} = \Hb(\ZZ_{\ell} ) $. 
Set      $$ \tau _{\ell}   \colonequals  \begin{pmatrix}    1   &    &  &   \scalebox{1.1}{$\frac{1}{\ell}$}     \\&  1 &   \scalebox{1.1}{$\frac{1}{\ell}$}    \\[0.5em] & &   1 & \\& & &   1 \end{pmatrix} \in \Gb(\QQ_{\ell}) $$
and denote by $ \tau $ the image of $ \tau_{\ell} $ under the  embedding  $ \Gb(\QQ_{\ell}) \hookrightarrow   \Gb(\Ab_{f}    )   $, so that  $ \tau  =  \tau^{\ell}   \tau_{\ell}  $  where $ \tau^{\ell}  \in  \Gb(\Ab_{f}^{\ell} )  $ is   identity.  The  convention   introduced  in   Notation  \ref{nota3} is  maintained.

\begin{lemma} \label{connC}  If  $ \ell =  2 $,   the  morphism     $ \iota \colon \Sh_{\Hb}( H_{\tau} ) \to  \Sh_{\Gb}(\tau K \tau^{-1}) $ is a closed immersion.    
\end{lemma}     
\begin{proof}  Let $  L  = L_{\ell} L^{\ell}  $ be a compact open subgroup of $ G $  
that  contains  both  $ K $ and $ wKw $. Clearly $ L_{\ell} = K_{\ell} $ by maximality of $ K_{\ell}  $.    
Write $ w = w_{\ell} w^{\ell} $ and set $  L '   \colonequals \tau  L  \tau^{-1} $. Then $ L' $ is neat and contains $ \tau K \tau  ^ { - 1  }   $.    By  Lemma  \ref{closedimmersion}, $ \iota = \iota_{H_{\tau},  \tau K \tau^{-1} } $ is a closed immersion whenever $ L'  $ contains $ w \tau K  \tau^{-1} w $. 
Note  that   $$   L '   = L ^{\ell}  \cdot    L'_ {\ell} \quad  \text{ where }  \quad L'_{\ell}   =  \tau_{\ell} K   _  {  \ell   }      \tau_{\ell}^{-1} . $$ Now 
$ L^{\ell}  $ 
contains 
$  w ^ { \ell } \tau^{\ell}   K^{\ell} \tau^{\ell}  w ^ { \ell }      = w^{\ell} K^{\ell} w^{\ell}  $ by  our  choice. 
So $ L' $ contains  
$ w \tau K \tau^{-1} w $ if and only if $ \tau_{\ell} K_{\ell} \tau_{\ell} ^{-1}  $ contains (and hence equal to)       $  w _{\ell}  \tau_{\ell} K_{\ell} \tau_{\ell}^{-1}   w_{\ell}  $,    
i.e., when    $$ \gamma_{\ell} \colonequals  \tau_{\ell}^{-1} w_{\ell}  \tau_{\ell} =   \left(\begin{matrix}
1 & &  & \scalebox{1.1}{$\frac{2}{\ell}$}     \\ 
 & -1 &    -   \scalebox{1.1}{$\frac{2}{\ell}$}     &  \\[0.5em]    
 &  & 1 &  \\
 & &  & -1
\end{matrix}\right)   
$$
normalizes $ K_{\ell}     $. This is true  
for $ \ell = 2 $, where we  even   have     $ \gamma_{\ell}  \in K_{\ell} $.       
\end{proof}

\begin{lemma}   \label{B_{i}}   For all   $ \eta \in \left \{ h , h \tau \, | \, h \in H \right \} $, we have $  \nu(H_{\eta}) = \nu(U) $.   
\end{lemma} 
\begin{proof} 
For any $ h  \in H $ and $ \eta \in G $, we have $ \nu(H_{h\eta}) = \nu(H_{\eta}) $ as     $ H_{h\eta} =  h H_{\eta} h^{-1} $.    
So it  suffices restrict   attention  to $ \eta \in \left \{ 1_{G}, \tau \right \}  $. 
The case $  \eta   = 1_{G}  $ is  trivial  and  the  argument for $ \eta =  \tau   $ is as  follows.    Let $  \mathbf{A}  \subset \Gb_{\QQ_{\ell}} $ denote the   maximal     diagonal  torus  and  let $  A_{\ell} \colonequals   \mathbf{A}(\QQ_{\ell} ) $,  $ A^{\circ}_{\tau_{\ell} }  =  A _{\ell}    \cap \tau    _ { \ell  }  K_{\ell} \tau _ {  \ell  }    ^{-1} $. Then $ \mathrm{diag} (a,b,c,d) \in A_{\tau_{\ell}}^{\circ} $ iff $$ a,b,c,d \in \ZZ_{\ell}^{\times},  \,  ac = bd,      a - d , b - c \in \ell \ZZ_{\ell} .  $$ 
It is then easily seen that $ \nu( A^{\circ}_{\tau_{\ell} }) = \ZZ_{\ell}^{\times} $. Since $  H_{\tau} $ contains $   
A^{\circ}_{\tau_{\ell} }U^{\ell} $, we have  $ \nu(H_{\tau})  = \widehat{\ZZ}^ { \times  }   =  \nu(U)  $. 
\end{proof}
Define $ \sigma = \sigma^{\ell}  \sigma_{\ell} \in \Gb(\Ab_{f} ) $ by setting $   \sigma_{\ell} = \mathrm{diag}(\ell, \ell ,1 , 1 )  \in \Gb(\QQ_{\ell}) $ and   $ \sigma^{\ell} = 1 $. 
\begin{lemma}   \label{deggsp4}      We have 
$ K \sigma K = U \sigma K \sqcup U \sigma \tau K  $. Moreover,  
\begin{enumerate}   [label=\normalfont (\alph*), before = \vspace{\smallskipamount}, after =  \vspace{\smallskipamount}]    \setlength\itemsep{0.1em}    
\item  $ | K \sigma K/ K | = 1  + \ell  + \ell  ^{2} +   \ell  ^{3}    $ 
\item    $  \deg \,  [U \sigma   K ]_{*}  =      (1 + \ell)^{2} $ 
\item       $ \deg \, [U \sigma  \tau    K ]_{*}  = 1 $  
\end{enumerate}
\end{lemma}  
\begin{proof}  The first statement is  \cite[Proposition 9.3.3]{AES}.  Alternatively, note that  by \cite[p.\ 38]{Taylor}  or  \cite[p.\ 189]{Schmidt}, a set of representatives for $ K \sigma  K  /  K  $ (in our convention)  is given by $$ 
\begin{pmatrix}   
1  & &  &   \\
& 1 & & \\
& &  \ell & \\
& & &  \ell 
\end{pmatrix}  ,   \quad    \begin{pmatrix}   
1 &  &   &   \\
& \ell &  & b \\
& & \ell & \\
& & &  1   
\end{pmatrix}  ,   \quad    \begin{pmatrix}   
 \ell  & a & b &   \\
& 1 & & \\
& & 1 & \\
& & -a & \ell 
\end{pmatrix}  ,  \quad       \begin{pmatrix}   
\ell  & &  b &  a  \\
&  \ell &  a  & c    \\
& & 1 & \\
& &  &   1 
\end{pmatrix}    ,    $$
where  the entries  $ a, b , c $  in each of the  displayed matrices  run over  $ 0 , 1,  \ldots,  \ell - 1   $. One now shows  by  applying elementary row and column operations  that the classes of these matrices in $ U \backslash  G /  K  $ are represented by $  \sigma $ and $ \sigma \tau $. That $ \sigma $ and $ \sigma \tau $ represent distinct    classes in $ U \backslash G   /   K   $       follows by noticing that $ HK \neq H \tau K $ (which one can see by showing that $ h \tau \notin K $ for any $ h \in H $).  

Part (a) follows from  the decomposition above. Part (b) follows by \cite[Lemma 9.4.1(a)]{AES} by evaluating the function computed there at the zero matrix and part (c) by Lemma 9.3.2 in \emph{loc.cit}. Alternatively, (b) can also be computed by relating  $   \deg  \,    [U \sigma K ]_{*}  $  to the degree of the $ T_{\ell} $ Hecke operator for $ \GL_{2} $ and part (c) by noticing that the conditions on the matrix entries of an element  $ h \in H$  imposed  by requiring $  h \in  (\sigma  \tau )^{-1} H  \sigma \tau  \cap  K $ or $ h \in  ( \sigma \tau )  ^{-1}  U   (\sigma \tau) \cap K $ are the same.  
\end{proof} 
Denote $ \sigma_{1} \colonequals \sigma $ and  $ \sigma_{2} \colonequals \sigma \tau $, so that $ I = \left \{ 1, 2 \right \} $ is our indexing set. 
By our choice of $ K $  and  Lemma  
  \ref{dg}, we have      $ d_{\sigma_{1}, K } = d_{1_{G} , K } =  1 $. Lemma \ref{B_{i}} and Lemma \ref{Bilemma}    
imply that we can  take  $ B_{i} = \left \{ 1_{H} \right \} $ and  that  $  c_{i} =  e_{i}    $ for $ i = 1, 2 $.  
 Invoking Lemma \ref{deggsp4},  equation   (\ref{mainfalse})   reads     
\begin{equation}   \label{falseequal}     1 +  \ell  + \ell  ^{2} +  \ell   ^{3}  \stackrel{   ? } { = }      (\ell +1)^{2} +  d_{\sigma_{2},   K    }  
\end{equation}    
If we have $ \ell = 2 $ (e.g., take $ K = K(N) $ for $ N \geq 3 $ odd), we have $ d_{\tau,   K } = 1 $   by     Lemma \ref{connC}  and so $ d_{\sigma_{2},  K } = 1 $  by    Lemma   \ref{dg}.       But in that case,  the LHS   of   (\ref{falseequal})    is $  15 $ while the      RHS is $ 10   $.    

\subsection{Unitary groups} In  \S \ref{mainfalseexampleone}, our eventual counterexample only worked for $ \ell = 2 $ which might seem a little unsatisfactory in terms of scope. 
In this section, we consider certain unitary Shimura varieties for which there is an abundance of elements $ w $ satisfying  Lemma  \ref{closedimmersion}.  This allows us to  furnish a set of counterexamples for  all  primes $ \ell $  that   are     split in an imaginary  quadratic extension used to define the Shimura variety.  Although the ideas are the same before, a little more work is involved.    

Let $E = \QQ(\sqrt{-d})$ be an imaginary quadratic field, and let $\gamma \in \Gal(E/\QQ)$ denote the nontrivial element. For an integer $p$, let $\mathrm{GU}(p,p)$ be the unitary similitude group over $\QQ$ of signature $(p,p)$, defined with respect to the Hermitian pairing over $E$ given by the Hermitian matrix $
J = \mathrm{diag}(1,-1,1,-1, \ldots, 1,-1) $.  
That is, $J$  is   the  $ 2p  \times   2p  $  diagonal matrix with $1$ in the odd-numbered entries and $-1$ in the even-numbered entries (cf.\:\cite[\S 3.1]{Anticyclo}).
Let  $ c \colon \mathrm{GU}(p,p) \to \GG_{m} $ denote  the   similitude map  and  $ \det : \mathrm{GU}(p,p)  \to  \mathrm{GU}(1) $ the  determinant.    Set    $$  \Hb = \mathrm{GU}(1,1) \times_{c}   \mathrm{GU}(1,1) ,      \quad \quad   \Gb = \mathrm{GU}(2,2) $$    and let $\iota \colon \Hb \to \Gb$ be the embedding $(h_{1}, h_{2}) \mapsto \mathrm{diag}(h_{1}, h_{2})$. Then both $\Hb$ and $\Gb$ admit standard Shimura data, given by sending $z \in \mathbb{C}^\times$ to alternating copies of $z$ and $\bar{z}$ along the diagonal. Moreover, both of these data satisfy (SV5), and $\iota$ constitutes an embedding of Shimura data. We also observe that the derived group of $\Hb$ is simply connected, since $(\Hb^{\mathrm{der}})_{\overline{\QQ}} \simeq \SL_2 \times \SL_2$. Therefore, Corollary~\ref{zkgcoro} applies in this context as well.    

Let $  \mathrm{U}    _{1} $ denote torus of norm one elements in $ E^{\times} $,  as in \cite[\S 3.1]{Anticyclo}. Then $ \Tb  =  \Hb / \Hb ^{\mathrm{der} }    $  is  
 isomorphic to $    \mathrm{U}_{1} \times \GG_{m} \times \mathrm{U}_{1}   $ in such a  way  that     $  \nu  \colon  \Hb  \to   \Tb   $  is  identified with  the map    \begin{align*}  \nu \colon  \Hb   &  \longrightarrow \mathrm{U}_{1} \times \GG_{m} \times \mathrm{U}_{1} ,   \\
h = (h_{1}, h_{2})    &   \longmapsto   \left ( \frac{ c(h) }  { \det  h_{1}       }  ,   c(h),   \frac{c(h) } { \det h_{2} }   \right   ) 
\end{align*}    
where $ c(h) \colonequals c(h_{1}) = c(h_{2}) $ is the  common   similitude. It  is  easily  seen      that     $ \Tb(\QQ) $ is discrete in $ \Tb(\Ab_{f}) $.    Let $ \mathscr{N}  \colon E^{\times} \to \QQ^{\times} $ denote the norm  map    and   let   $     S \subset E^{\times}  \setminus \left \{ 1 \right \}  $ denote set of units such that   $ \mathscr{N} (\xi) = \xi \cdot  \gamma(\xi) =  1 $. For any $ \xi \in S $,   define    $$ w_{\xi}  \colonequals \mathrm{diag}(1,1,\xi,\xi)  \in   \mathbf{Z}_{\Hb}(\QQ) . $$
For any $ \QQ $-algebra $ R$ and $ g \in \Gb(R) $,  the   condition  $ g w_{\xi} = w_{\xi} g $  forces $ g $ to be block diagonal. Thus     the centralizer of $ w_{\xi}$ in $ \Gb $ equals  $  \Hb   $. As before,  we  will   consider $ w_{\xi} \in \Gb(\QQ) $ as an element of $ G $  via the diagonal embedding $    \Gb(\QQ)  \hookrightarrow \Gb(\Ab_{f}) =  G $.  

Let $ \ell \nmid 2  d      $ be a  fixed  rational prime  that is split in $ E $.  If $ \jmath \colon E \to \QQ_{\ell} $ is an embedding  and    $ \xi \in S $ is such that $ \jmath(\xi) \in 1 + \ell \ZZ_{\ell} $, then $ \jmath \circ \gamma (\xi) \in 1 + \ell \ZZ_{\ell} $ as well, since $     \gamma(\xi) $ is the inverse of $ \xi   $. We will refer to such $ \xi \in S $ as  \emph{$ \ell $-invertible}.          For $  m   $ a positive integer, $$ \xi_{m}   \colonequals  \left (   \frac{ 1 - d m  ^{2}  \ell ^{2} } { 1 +  d  m^{2}  \ell^{2}}  \right )   +  \left   (     \frac{   2      m \ell }   {   1  + d  m^{2 }  \ell  ^  {  2   }    }  \right ) \sqrt { -d  }     \in S   $$ 
are examples of such elements.

Fix now a neat compact open subgroup $K$ and let $\ell$ be a prime split in $E$ such that $K$ is hyperspecial at $\ell$, i.e., $K = K^{\ell} K_{\ell}$ with $K_{\ell} = \Gb(\ZZ_{\ell})$. For any $\xi \in S$, define $K_{\xi} := K \cap w_{\xi}^{-1} K w_{\xi}$. Then $K_{\xi}$ is a neat compact open subgroup of $G$, and both $K_{\xi}$ and $w_{\xi} K_{\xi} w_{\xi}^{-1}$ are contained in $K$. If $\xi$ is also $\ell$-invertible, then $\xi \in E \otimes_{\ell} \QQ_{\ell} \simeq \QQ_{\ell} \oplus \QQ_{\ell}$ lies in $\ZZ_{\ell}^{\times} \times \ZZ_{\ell}^{\times}$ and the $\ell$-component $w_{\xi,\ell} \in \Gb(\QQ_{\ell})$ of $w_{\xi}$ lies in $K_{\ell}$. Thus, for  an  $\ell$-invertible  element   $\xi \in S $, which we fix in what follows, $K_{\xi}$ is hyperspecial at $\ell$.
The upshot of this discussion is that, by replacing $K$ with $K_{\xi}$, we may assume that

\begin{itemize} 
\item $ K $ and $w_{\xi}  K w_{\xi}^{-1} $ are hyperspecial (and equal) at $ \ell $,   
\item $ K $ and $  w_{\xi} K  w_{\xi}^{-1} $ are contained in   a     common     neat compact open subgroup.
\end{itemize}  
Set $ U \colonequals K \cap \Hb(\Ab_{f}) $.    Then $ \iota_{U,K} $ is a closed immersion by   Lemma  \ref{closedimmersion} and our assumptions   on $ K  $.   Let   
   
$$ \tau _{\ell}   \colonequals       \begin{pmatrix}    1   &    &  \scalebox{1.1}{$\frac{1}{\ell}$}     \\&  1 &    
\\[0.5em] & &   1 & \\& & &   1 \end{pmatrix} \in \Gb(\QQ_{\ell}) $$
and let $ \tau = \tau_{\ell} \tau^{\ell}  \in  \Gb(\Ab_{f}    )   $ where $ \tau^{\ell} $ is identity.

\begin{lemma}  \label{connCforxi}  The  morphism      $ \iota \colon  \Sh_{\Hb}(H_{\tau} )  \to  \Sh_{\Gb} ( \tau K \tau^{-1} ) $ is a closed  immersion.   
\end{lemma}

\begin{proof}    Arguing  analogously to Lemma  \ref{connC},  the argument boils down to showing that  $$ \tau_{\ell}^{-1} w_{\xi,\ell} \tau_{\ell}  =    \left(\begin{matrix}
1 & & \scalebox{1.1}{$\frac{1-\xi}{\ell}$}     \\ 
 & 1 &   
 \\[0.5em]    
 &  & \xi &  \\
 & &  &  \xi 
\end{matrix}\right)   \in  \Gb ( \QQ_{\ell} )       $$  
lies in $ K_{\ell} $, which it does by $ \ell $-invertibility    of $ \xi $.          
\end{proof}  
\begin{lemma}  \label{B_{i,xi}}    For all $  \eta  \in \left \{ h , h \tau \, | \, h \in 
 H    \right \} $, $  \nu(H_{\eta}) =   \nu(U)    $.         
\end{lemma}    
\begin{proof} As in Lemma \ref{B_{i}},   it  suffices  to restrict to $  \eta  \in \left \{ 1_{H}   , \tau \right \} $ and the   claim for $  \eta =  1_{H}  $ is  again  trivial.    At a split prime, the choice of  an   $ \alpha \in \QQ_{\ell} $ such that $ \alpha^{2} = -d $ determines   compatible       isomorphisms $ \Tb_{\QQ_{\ell}}  \simeq   \GG_{m}^{3} $ 
and $    \Hb_{\QQ_{\ell}}  \simeq \GG_{m} \times \GL_{2}  \times  \GL_{2}  $    
such that $ \nu \colon \Hb_{\QQ_{\ell}}  \to \Tb_{\QQ_{\ell}}  $ is given by $ (c, h_{1}, h_{2})   \mapsto   \left  (  c \det h_{1}^{-1} , c , c  \det h_{2} ^{-1} \right ) $.    
Then one can show that $ H_{\tau_{\ell}} \colonequals \Hb(\QQ_{\ell}) \cap \tau_{\ell} K _ { \ell  }    \tau_{\ell}^{-1} $ contains  the  subgroup  $$  \left \{ \big (u, \mathrm{diag}(a, b),\mathrm{diag}(a,c)\big )  \in  \GG_{m}  \times  \GL_{2}  \times \GL_{2} \, | \, a,b,c,u \in \ZZ_{\ell}^{\times} \right \}  . $$   From this, one deduces   that     $ \nu(H_{\tau_{\ell}}) =   \ZZ_{\ell} ^ {  \times }   \times \ZZ_{\ell}^{\times}  \times   \ZZ_{\ell}^{\times } =  \nu(U   _ {  \ell    }     )  $. 
\end{proof}

Since $ \ell $ is split in $ E $, we can fix isomorphisms $\Hb_{\QQ_{\ell}} \simeq \GG_{m} \times \GL_{2} \times \GL_{2}$ and $\Gb_{\QQ_{\ell}} \simeq \GG_{m} \times \GL_{4}$ so that  the  local  embedding   $\iota \colon \Hb_{\QQ_{\ell}} \hookrightarrow  \Gb_{\QQ_{\ell}}$ is identified with the embedding \begin{align*} \GG_{m} \times \GL_{2} \times \GL_{2} &  \hookrightarrow \GG_{m} \times \GL_{4}  \\
(c, h_{1}, h_{2})   &  \mapsto (c,   \mathrm{diag}(h_{1},  h_{2} )    )   .    
\end{align*}
  Let $  \sigma $, $ \sigma ' \in G $ be defined so that their components away $ \ell $ are identity and at $ \ell $  are     given by  $ \sigma  _ { \ell }   \colonequals \big  (1,    \mathrm{diag}(\ell,1,1,1) \big ) $, $ \sigma'_{\ell} =  \big (1, \mathrm{diag}(1,1,  \ell, 1 ) \big ) $. 
\begin{lemma}  We have $ K \sigma K = U \sigma K \sqcup  U \sigma' K \sqcup  U \sigma \tau K $.   Moreover,     
\begin{enumerate} [label=\normalfont (\alph*), before = \vspace{\smallskipamount}, after =  \vspace{\smallskipamount}]    \setlength\itemsep{0.1em}   
\item  $ | K \sigma K / K |  =  1 + \ell  +    \ell^{2} + \ell^{3} $,  
\item   $  \deg   \,  [U \sigma  K ]_{*}  = \deg   \,   [U \sigma ' K ]_{*}     =  
1  +   \ell
  $, 
\item    $  \deg  \,     [ U \sigma \tau  K] _{*}     = 1   $    
 
\end{enumerate}
\end{lemma}   
\begin{proof}  The first claim is a special case of \cite[Proposition 7.4.5]{AES}.    Alternatively, note that the stratification of the  projective space  $  \mathbb{P}^{3} $ over the field  $ \ZZ / \ell \ZZ $ implies that a set of representatives for $  K \sigma K /  K   $ is  given by   $$ 
\begin{pmatrix}   
1  & &  &   \\
& 1 & & \\
& &  1 & \\
& & &  \ell 
\end{pmatrix}  ,   \quad    \begin{pmatrix}   
1 &  &   &   \\
&  1  &  & \\
& & \ell &  a \\
& & &  1  
\end{pmatrix}  ,   \quad    \begin{pmatrix}   
1  &  &  &   \\
&  \ell & a  & b  \\
& & 1 & \\
& & &   1    
\end{pmatrix}  ,  \quad       \begin{pmatrix}   
\ell  & a &  b &  c  \\
&   1 &   &    \\
& & 1 & \\
& &  &   1 
\end{pmatrix}    ,  $$ 
where the $ \QQ_{\ell}^{\times} $-component (in $  \Gb(\QQ_{\ell}) \simeq  \QQ_{\ell}^{\times}  \times \GL_{4}(\QQ_{\ell}) $)  is taken to be   $ 1 $, and   where  the  entries    $ a, b, c $  in each of  the  displayed matrices  run over $  0 , 1, \ldots, \ell - 1 $.  
One then reduces each of these matrices by  appropriate  row and column operations to show that the classes of these matrices in $ U \backslash G /  K $  are represented by $ \sigma , \sigma   '    $ and $ \sigma \tau $. It is easily checked that $ H K \neq H \tau K $  which distinguishes $ U \sigma \tau K $ from $ U \sigma K $ and $ U \sigma ' K  $.  To  distinguish $ U \sigma K $ from $ U \sigma ' K  $,  we use the Cartan decomposition for the  double  quotient  $ \Hb(\ZZ_{\ell})  \backslash \Hb(\QQ_{\ell})  /  \Hb(\ZZ_{\ell})  $ and  an elementary  trick  established  in  
  \cite[Lemma 4.9.2]{AES}.   
  
Part (a) follows from   the  decomposition above. Part (b) and (c)  can  be   deduced  along the lines  outlined in    the  proof of  Lemma   \ref{deggsp4}.    
\end{proof}    
Set $ \sigma_{1}  \colonequals  \sigma $, $ \sigma_{2} \colonequals  \sigma ' $, $ \sigma_{3} \colonequals  \sigma \tau $,  so 
 that  $ I    =  \left  \{ 1 , 2,  3  \right  \} $ is  our  indexing set.    
We have $   d_{1_{G}, K  }    =      d_{\sigma_{i}, K } = 1 $ for $ i = 1 ,2   ,     3 $ by  
our choice of $ K $, Lemma   \ref{connCforxi}  and  Lemma       \ref{dg}. Lemma  \ref{B_{i,xi}} (in  conjunction with   Lemma   \ref{Bilemma})    implies  that   we may take $ B_{i} = \left \{ 1_{H} \right \} $ 
and  that  $   c_{i} = e_{i}    $ for $ i = 1, 2, 3 $.    Equation (\ref{mainfalse})    then    reads $$  1 + \ell  +  \ell^{2} +  \ell^{3}   \stackrel{?}{=}      ( 1 + \ell )  +  ( 1 + \ell)  + 1  $$
which is false for all $ \ell $.
\label{mainfalseexampletwo}  

\section{Schwartz spaces}   \label{Schwartzsection}  To have another and a technically    simpler    perspective on the failure of (\ref{falsezkg}) in general,  we investigate an auxiliary  question in the setting of Schwartz spaces motivated by the discussion in \S \ref{specialcycles}.    This section can be read independently of the rest of this    note.  

Let  $ G $ be  a locally profinite group and $ J $  be   a closed subgroup of $ G $. Then $ X \colonequals J \backslash G $ 
is a locally compact Hausdorff and totally  disconnected topological space with a continuous right action $ X \times G \to X $ given by $ (J\gamma , g)    \mapsto J \gamma g $.
Let $ \mathcal{S}_{\ZZ}(X) $ denote the   $ \ZZ $-module  of  all $ \ZZ$-valued  locally constant compactly supported functions on $ X    $.  We have an induced smooth left action  \begin{equation}  \label{natgact}  G \times \mathcal{S} _{\ZZ}(X)   \to    \mathcal{S}_{\ZZ}(X),  \quad \quad   
(g, \xi    )      \mapsto  \xi ( (  -  ) g )  
\end{equation} 
For each compact open subgroup $ K $ of $ G $, we let $ M(K) \colonequals \mathcal{S}_{\ZZ}(X/ K) = \mathcal{S}_{\ZZ}(X)^{K} $ denote the $ \ZZ $-module  of all  functions in $ \mathcal{S}_{\ZZ}(X) $ that are $ K $-invariant under the natural $ G $-action. 
For $ \sigma \in G $, we define  $  \ZZ  $-linear  maps    
\begin{alignat*}{4} \mathcal{T}(\sigma) \colon M(K) &  \to   M(K) &  \hspace{0.5in}      [K \sigma K ]_{*}  \colon M(K)  &     \to  M(K) \\ 
\ch (J g K )  &  \mapsto  \sum_{ \gamma \in K \sigma K / K }  \ch ( J g \gamma K )  &  \hspace{0.5in}  \ch(J g  K ) &     \mapsto   \sum _ {  \delta \in   K  \backslash  K \sigma K }  \ch ( J g K  \delta ) .
\end{alignat*}  
where $ \ch  (    -    )  $ denotes the characteristic function. Then $ \mathcal{T}(\sigma) $ and $ [K \sigma K]_{*} $ are respectively the analogs of (\ref{CM1}) and  (\ref{CM2}) in this setting.  
\begin{question}   \label{hypoquestion}        Is $ \mathcal{T}(\sigma) =  [K \sigma K]_{*} $ for all $ \sigma \in G $?     
\end{question}

Of course, this is the case when $ J $ is the trivial  subgroup,   
since    both $ \mathcal{T}(\sigma) $ and $ [K \sigma K]_{*} $ send  $    \ch(g K$) to the characteristic  function   of $  g    K \sigma K   $.   More generally, we  have the  following.        

\begin{proposition}   \label{hypoworks}  Suppose that for any $ \gamma  \in G $, $ J \cap \gamma K \gamma ^{-1}  $ equals a fixed subgroup of $ J  $.   Then for all $ \sigma \in G $,  $ \mathcal{T}(\sigma) = [K \sigma K ]_{*} $.  
\end{proposition}   
\begin{proof} Let $ \Sigma $ be the collection of  all   compact  open  subgroups of $ G $ that are equal to a finite intersection of conjugates of $ K $. 
Our assumption on $ J $  implies that for any $ g \in G $ and  $ L \in \Sigma $ satisfying  $ L \subset K $,   
$ J g K = \bigsqcup _ { \mu \in K / L } J g \mu L. $ If we take $ L $ to be the intersection of $   K $ and $  \bigcap_{\delta \in  K  \backslash  K \sigma  K }   \delta  ^ { - 1  }  K \delta  $, then  $  L $ is normal in $ K $ and its conjugates by $ \delta \in K \sigma K $ are contained in $ K $. These properties imply that   for any $ \gamma , \delta \in K \sigma K $ and $ g \in G $,  
\begin{align*} \ch ( J  g \gamma K )  &  =  \sum   \nolimits  _ { \mu \in K / L }  \ch (  J g  \gamma \mu L )\\
\ch ( J   g    K  \delta  ) &  =  \sum   \nolimits      _ {  \nu \in \delta^{-1} K  \delta / K }  \ch ( J g   \delta    \nu      L )
\end{align*} 
We wish to  show that $$  \sum _ { \gamma \in K \sigma K / K }  \sum _ { \mu \in K / L } \ch (  J g  \gamma  \mu L )  =   \sum _  { \delta  \in  K  \backslash  K \sigma K  }  \sum _ { \nu \in \delta ^{-1} K   \delta / L }  \ch(  J  g \delta  \nu L ) . $$   
Both sides are sums of characteristic function on cosets $ J \backslash G / L $,  possibly with  repetitions. But both the  lists
\begin{itemize}    
\item $ \gamma \mu L $ where  $ \gamma \in K \sigma K / K $ and $ \mu \in K / L  $,         \item $  \delta \nu L   $ where  $  \delta \in   K   \backslash  K  \sigma K $ and    $ \nu  \in \delta ^{-1}  K \delta / L    $       
\end{itemize} 
of cosets in $ G / L $ enumerate  each element of  $ K \sigma K / L $  exactly 
 once.    This  proves the  claim.  
\end{proof}  
An example where the condition of Proposition  \ref{hypoworks} is  satisfied  is when $ G = \Gb(\Ab_{f}) $ for some reductive group $ \Gb $ over $ \QQ$, $ K $ is a neat compact open subgroup of $ G $  and  $ J = \Tb (\QQ) $ where $ \Tb \hookrightarrow \Gb $ is a torus in $ \Gb $ such that $ \Tb (\QQ) $ is discrete in $ \Tb(\Ab_{f}) $. Indeed, $  \Tb(\Ab_{f})  \cap \gamma K \gamma ^{-1} $ is a compact open subgroup of $ \Tb(\Ab_{f}) $ for any $ \gamma   \in    G $. So the discreteness of $ J =  \Tb(\QQ) $ in $ \Tb(\Ab_{f}) $ implies that $ J \cap \gamma K \gamma^{-1} $ is finite and the neatness of $ K  $  forces $  J \cap  \gamma K \gamma^{-1}  $  to  be torsion free   (see \cite[Definition B.6] {Anticyclo}).        Cf.\   Example      \ref{mixedCM}.       It is however also   easy to find situations where $ \mathcal{T}(\sigma) \neq [K \sigma K ]_{*}  $.  
\begin{example} Suppose  $ J = K$. Then $ M(K) = \ZZ [ K \backslash G / K ] $ is the $ \ZZ $-module of characteristic functions of double cosets in $ K \backslash G  / K $.   Now for any $ g \in G $,  
$ \mathcal{T}(\sigma)  \cdot  \ch(K  g K) = \sum_{ \gamma \in K \sigma K /  K }   \ch (K g \gamma K ) $ by definition whereas $$  
[K \sigma K]_{*}  \cdot  \ch(K g K ) = \ch(K g K ) * \ch ( K \sigma K )  $$
where $ * $ denote the convolution product on $ M(K) $ 
with respect to a Haar measure  on $  G $  $ \mu $ such that $ \mu(K) = 1 $. Note that the map  $ \mathrm{ind} \colon M(K) \to   \ZZ $ given by $      
\ch( K  \gamma  K )   \mapsto | K \gamma K / K | $ is a $ \ZZ $-algebra homomorphism (see, e.g., \cite[\S 2.3]{AES}), so $ \mathcal{T}( \sigma ) \cdot \ch (K g K )  =   [ K \sigma K ] _{*}  \cdot  \ch ( K g K ) $
would imply that   
\begin{equation}  \label{falsedeg}   \sum  _ { \gamma \in K \sigma K  / K  } |K  g \gamma K / K | = | K g K /  K    | \cdot | K \sigma K /K | .  
\end{equation}     
This is clearly false in general. For instance, if $ g = 1_{G} $, the LHS of (\ref{falsedeg})  is $ | K \sigma K / K |^{2}  $ while  the  RHS  is  $ | K \sigma K / K | $ and these are not equal as soon as $ | K \sigma  K / K  |  \neq   1    $.  
\end{example}    

\begin{remark}   \label{motivation}    We may  consider   Question   \ref{hypoquestion} as a problem of ``explicit descent''  in the  following  sense.    We know that $  \sum _ {  \delta \in    K  \backslash  K \sigma K  }  \ch ( J  g   K  \delta ) $ is an  element of $ M(K) $ as it is $ K $-invariant (a descent phenomenon) and we  even  know that its support  is   $ J  g   K \sigma K $. We however want an \emph{explicit}  linear combination of the basis $  \left \{  \ch(J \gamma    K ) \, | \,  \gamma \in G \right \}  $ of $ M(K) $ that equals this element. This  amounts to computing certain volumes of subsets of $ J $.   
If one  establishes  that  (\ref{Phi}) or an appropriate variant of it    is  a bijection and the maps are equivariant for varying $ K$,  then   one has an  alternate strategy for deriving       Corollary \ref{zkgcoro}.   This approach however  does not work in  the generality of Theorem \ref{descent}.  
\end{remark}  
\bibliographystyle{amsalpha}    
\bibliography{refs}
\Addresses

\end{document}